\documentclass[12pt]{amsart}
\usepackage{}
\usepackage{amscd,amsmath,amsthm,amssymb}
\usepackage{pstcol,pst-plot,pst-3d}%\usepackage[T1]{fontenc}
\usepackage{color}
\usepackage{pstricks}
\usepackage{lineno,xcolor}
\DeclareSymbolFont{largesymbol}{OMX}{yhex}{m}{n}
\DeclareMathAccent{\Widehat}{\mathord}{largesymbol}{"62}

\newpsstyle{fatline}{linewidth=1.5pt}
\newpsstyle{fyp}{fillstyle=solid,fillcolor=verylight}
\definecolor{verylight}{gray}{0.97}
\definecolor{light}{gray}{0.9}
\definecolor{medium}{gray}{0.85}
\definecolor{dark}{gray}{0.6}

 \def\G{{\mathcal G}}

 \def\supp{{\mathcal supp}}

 \def\opn#1#2{\def#1{\operatorname{#2}}} % to make operators
 \opn\chara{char} \opn\length{\ell} \opn\pd{pd} \opn\rk{rk}
 \opn\projdim{proj\,dim} \opn\injdim{inj\,dim} \opn\rank{rank}
 \opn\depth{depth} \opn\grade{grade} \opn\height{height}
 \opn\embdim{emb\,dim} \opn\codim{codim}
 
 \opn\Tr{Tr} \opn\bigrank{big\,rank}
 \opn\superheight{superheight}\opn\lcm{lcm}
 \opn\trdeg{tr\,deg}%\emph{
 \opn\reg{reg} \opn\lreg{lreg} \opn\ini{in} \opn\lpd{lpd}
 \opn\size{size} \opn\sdepth{sdepth}
 \opn\link{link}\opn\fdepth{fdepth}\opn\lex{lex}
 \opn\tr{tr}
 \opn\type{type}
 \opn\Borel{Borel}
\opn\cdeg{cdeg}

 \opn\div{div} \opn\Div{Div} \opn\cl{cl} \opn\Cl{Cl}
 \opn\Spec{Spec} \opn\Supp{Supp} \opn\supp{supp} \opn\Sing{Sing}
 \opn\Ass{Ass} \opn\Min{Min}\opn\Mon{Mon}
 \opn\Ann{Ann} \opn\Rad{Rad} \opn\Soc{Soc}
 \opn\Im{Im} \opn\Ker{Ker} \opn\Coker{Coker} \opn\Am{Am}
 \opn\Hom{Hom} \opn\Tor{Tor} \opn\Ext{Ext} \opn\End{End}
 \opn\Aut{Aut} \opn\id{id}
 
 \opn\nat{nat}
 \opn\pff{pf}%   \pf exists already
 \opn\Pf{Pf} \opn\GL{GL} \opn\SL{SL} \opn\mod{mod} \opn\ord{ord}
 \opn\Gin{Gin} \opn\Hilb{Hilb}\opn\sort{sort}
 \opn\PF{PF}\opn\Ap{Ap}
 \opn\aff{aff} \opn
 \con{conv} \opn\relint{relint} \opn\st{st}
 \opn\lk{lk} \opn\cn{cn} \opn\core{core} \opn\vol{vol}  \opn\inp{inp} \opn\nilpot{nilpot}
 \opn\link{link} \opn\star{star}\opn\lex{lex}\opn\set{set}
 \opn\width{wd}
 \opn\Fr{F}
 \opn\QF{QF}
 \opn\G{G}
 \opn\type{type}\opn\res{res}
 \opn\gr{gr}
 
  \def\cdeg{deg}

 \def\pot#1#2{#1[\kern-0.28ex[#2]\kern-0.28ex]}

 \opn\dirlim{\underrightarrow{\lim}}
 \opn\inivlim{\underleftarrow{\lim}}
 \def\Implies{\ifmmode\Longrightarrow \else
         \unskip${}\Longrightarrow{}$\ignorespaces\fi}
 \def\implies{\ifmmode\Rightarrow \else
         \unskip${}\Rightarrow{}$\ignorespaces\fi}
 \def\iff{\ifmmode\Longleftrightarrow \else
         \unskip${}\Longleftrightarrow{}$\ignorespaces\fi}

 \let\:=\colon
 \newtheorem{Theorem}{Theorem}[section]
 \newtheorem{Lemma}[Theorem]{Lemma}
 \newtheorem{Corollary}[Theorem]{Corollary}
 \newtheorem{Proposition}[Theorem]{Proposition}

 \newtheorem{Example}[Theorem]{Example}
 
 \newtheorem{Definition}[Theorem]{Definition}

 \let\epsilon\varepsilon
 \let\kappa=\varkappa
 \def\qed{\ifhmode\textqed\fi
       \ifmmode\ifinner\quad\qedsymbol\else\dispqed\fi\fi}
 \def\textqed{\unskip\nobreak\penalty50
        \hskip2em\hbox{}\nobreak\hfil\qedsymbol
        \parfillskip=0pt \finalhyphendemerits=0}
 \def\dispqed{\rlap{\qquad\qedsymbol}}
 
 \opn\dis{dis}
 \def\pnt{{\raise0.5mm\hbox{\large\bf.}}}
 
 \opn\Lex{Lex}

\begin{document}
%\linenumbers
\title {Regularity of  powers of edge ideals of vertex-weighted  oriented unicyclic graphs}

\author {Guangjun Zhu$^{^*}$\!\!\!,  Hong Wang,  Li Xu and Jiaqi Zhang }

\address{Authors¡¯ address:  School of Mathematical Sciences, Soochow
University, Suzhou 215006, P.R. China}
\email{zhuguangjun@suda.edu.cn(Corresponding author:Guangjun Zhu),
\linebreak[4] 1240470845@qq.com(Li Xu), 651634806@qq.com(Hong Wang),\nolinebreak[2] zjq7758258@vip.qq.com(Jiaqi Zhang).}

\dedicatory{ }

\begin{abstract}
In this paper we provide some exact formulas for the  regularity of powers
of edge ideals of vertex-weighted oriented cycles and  vertex-weighted  unicyclic graphs. These formulas are  functions of the weight of  vertices and the number of edges. We  also give some examples to show that these  formulas are related to direction selection and
 the weight of vertices.
\end{abstract}

\thanks{* Corresponding author}

%		13H10   	Special types (Cohen-Macaulay, Gorenstein, Buchsbaum, etc.)
%		13D02   	Syzygies, resolutions, complexes
%		05E40   	Combinatorial aspects of commutative algebra
%		16S36   	Ordinary and skew polynomial rings and semigroup rings

%		14M25   	Toric varieties, Newton polyhedra [See also 52B20]
%		13A02   	Graded rings
%		13F20   	Polynomial rings and ideals; rings of integer-valued polynomials
%		13A18   	Valuations and their generalizations
%		06A11   	Algebraic aspects of posets


\maketitle

\setcounter{tocdepth}{1}
%\tableofcontents

\section{Introduction}

A {\em directed graph} or {\em digraph} $D$ consists of a finite set $V(D)$ of vertices, together
with a collection $E(D)$ of ordered pairs of distinct points called edges or
arrows. If $\{u,v\}\in E(D)$ is an edge, we write $uv$ for $\{u,v\}$, which is denoted to be the directed edge
where the direction is from $u$ to $v$ and $u$ (resp. $v$) is called the {\em starting}  point (resp. the {\em ending} point).
Given any digraph $D$, we can associate a graph $G$ on the same vertex set
simply by replacing each arrow by an edge with the same ends. This graph is called the
underlying graph of $D$, denoted by $G(D)$. Conversely, any graph $G$ can be regarded as
a digraph, by replacing each of its edges by just one of the  two oppositely oriented arrows with the same ends.
 Such a digraph is called an {\em  orientation} of $G$.
An orientation of a simple graph is referred to as an {\em simple  oriented}  graph.

Edge ideals of edge-weighted graphs were introduced and studied by Paulsen
and Sather-Wagstaff \cite{PS}. In this work we consider edge ideals of graphs which are
oriented and have weights on the vertices. In what follows by a weighted oriented
graph we shall always mean a vertex-weighted oriented graph.

 A vertex-weighted oriented graph is a triplet $D=(V(D), E(D),w)$, where $V(D)$ is the  vertex set,
$E(D)$ is the edge set and $w$ is a weight function $w: V(D)\rightarrow \mathbb{N}^{+}$, where $N^{+}=\{1,2,\ldots\}$.
Some times for short we denote the vertex set $V(D)$ and edge set $E(D)$
by $V$ and $E$ respectively.
The weight of $x_i\in V$ is $w(x_i)$, denoted by $w_i$  or $w_{x_i}$.
The edge ideal of a vertex-weighted digraph was first introduced by Gimenez et al \cite{GBSVV}. Let $D=(V,E,w)$ be a  vertex-weighted digraph with the vertex set $V=\{x_{1},\ldots,x_{n}\}$. We consider the polynomial ring $S=k[x_{1},\dots, x_{n}]$ in $n$ variables over a field $k$. The edge ideal of $D$,  denoted by $I(D)$, is the ideal of $S$ given by
$$I(D)=(x_ix_j^{w_j}\mid  x_ix_j\in E).$$

Edge ideals of weighted digraphs arose in the theory of Reed-Muller codes as initial ideals of vanishing ideals
of projective spaces over finite fields \cite{MPV,PS}.
If a vertex $x_i$ of $D$ is a source (i.e., has only arrows leaving $x_i$) we shall always
assume $w_i=1$ because in this case the definition of $I(D)$ does not depend on the
weight of $x_i$. If  $w_j=1$ for all $j$, then $I(D)$ is the edge ideal of its underlying graph.

Our motivation to study the regularity of powers of edge ideals springs from a famous result:
for  any homogeneous ideal $I$ in a polynomial ring, it is well known that the regularity of $I^t$ is asymptotically a linear function in $t$, that is, there exist constants $a$ and $b$ such that for all $t\gg 0$, $\mbox{reg}\,(I^t)=at+b$ (see \cite{CHT}). Generally, the problem of finding the exact linear form $at+b$ and the
smallest value $t_0$ such that $\mbox{reg}\,(I^t)=at+b$ for all $t\geq t_ 0$ has proved to be very difficult.  There are few classes of graphs for which $a$, $b$ and $t_0$ are explicitly computed (see \cite{AB,ABS,B1,BBH1,BHT,MSY}).

Our objective in this paper is to find $a$, $b$ and $t_0$ in terms of combinatorial
invariants of the vertex-weighted digraph $D$ when $D$ is a vertex-weighted oriented  unicyclic graph.
 The digraph $D=(V(D),E(D),w)$ is called  an oriented  unicyclic graph, denoted by $D=C_m\cup(\bigcup\limits_{j=1}^{s}T_j)$, if its underlying graph is $G=G_0\cup(\bigcup\limits_{j=1}^{s}G_j)$,  and $C_m$ is  an oriented cycle with underlying graph $G_0$ and $T_j$ is an oriented tree with  underlying graph $G_j$, its  orientation  is as follows:  if $V(G_0)\cap V(G_j)=\{x_{i_j}\}$, then $x_{i_j}$ is the root of $T_j$, and  all edges in $T_j$ are oriented away from $x_{i_j}$ for $1\leq j\leq s$.
In \cite{Z4}, the first three authors  derive some exact formulas for the  regularity of  edge ideals of vertex-weighted rooted forests and oriented cycles.
In \cite{Z5}, we provide some exact formulas for the  regularity of   powers of  edge ideals of vertex-weighted rooted forests.
To the best of our knowledge,  few papers consider the regularity of $I(D)^t$ for a  vertex-weighted digraph.

In this article, we are interested in algebraic properties corresponding to the   regularity of $I(D)^t$ for some  vertex-weighted oriented graphs. By using the approaches of Betti splitting and polarization, we derive some exact formulas  for the  regularity  of powers of edge ideals of some directed graphs.
The results are as follows:

\begin{Theorem}
Let $C_n=(V(C_n),E(C_n),w)$ be a vertex-weighted oriented cycle with $w(x)\geq 2$ for any $x\in V(C_n)$,   then  for any $t\geq 1$
$$\mbox{reg}\,(I(C_n)^{t})=\sum\limits_{x\in V(C_n)}w(x)-|E(C_n)|+1+(t-1)(w+1)$$
where $w=\mbox{max}\,\{w(x)\mid x\in V(C_n)\}$.
\end{Theorem}

\begin{Theorem}
Let $D=(V(D),E(D),w)$ be a vertex-weighted oriented unicyclic graph with $w(x)\geq 2$ for any $d(x)\neq 1$, then  for any $t\geq 1$
$$\mbox{reg}\,(I(D)^{t})=\sum\limits_{x\in V(D)}w(x)-|E(D)|+1+(t-1)(w+1)$$
where $w=\mbox{max}\,\{w(x)\mid x \in V(D)\}$.
\end{Theorem}

Our paper is organized as follows. In section  $2$, we recall some
definitions and basic facts used in the  paper.
In section $3$, we provide a special order on the set of minimal
monomial generators of powers of edge ideals of vertex-weighted oriented cycles. Using this order, we give exact formulas for the
regularity of powers of edge ideals of vertex-weighted oriented cycles in  section $4$.  Moreover, we give some examples to show regularity of powers of edge ideals of vertex-weighted oriented cycles is related to
direction selection and  the assumption  that $w(x)\geq 2$  for any vertex $x$ cannot be dropped.
In section $5$,  we  give some exact formulas for  regularity of powers of  edge
ideals of vertex-weighted oriented unicyclic graphs. Moreover, we also give some examples to show  regularity of powers of
 edge ideals of vertex-weighted oriented unicyclic graphs are related to
direction selection and  the assumption  that
$w(x)\geq 2$  if $d(x)\neq 1$ cannot be dropped.

\medskip
For all unexplained terminology and additional information, we refer to \cite{JG} (for the theory
of digraphs), \cite{BM} (for graph theory), and \cite{BH,HH2} (for the theory of edge ideals of graphs and
monomial ideals).  We greatfully acknowledge the use of  computer algebra system CoCoA (\cite{Co}) for our experiments.

Throughout this paper, if $C_n=(V(C_n),E(C_n),w)$ be  an $n$-cycle such that $w(x)\geq 2$ for any $x\in V(C_n)$,
we  set $x_j=x_i$ if
$j\equiv i$ mod $n$ $(1\leq i\leq n)$. The  oriented  unicyclic graph $D=C_m\cup(\bigcup\limits_{j=1}^{s}T_j)$ satisfying if its underlying graph is $G=G_0\cup(\bigcup\limits_{j=1}^{s}G_j)$,  and $C_m$ is  an oriented cycle with underlying graph $G_0$ and $T_j$ is an oriented tree with  underlying graph $G_j$, its  orientation  is as follows:  if $V(G_0)\cap V(G_j)=\{x_{i_j}\}$, then $x_{i_j}$ is the root of $T_j$, and  all edges in $T_j$ are oriented away from $x_{i_j}$ for $1\leq j\leq s$.

\medskip
\section{Preliminaries }

In this section, we gather together needed  definitions and basic facts, which will
be used throughout this paper. However, for more details, we refer the reader to \cite{BBH1,BM,FHT,HH2,JG,MPV,PRT,Z2,Z4}.

Every concept that is valid for graphs automatically applies to digraphs too.
For example, let $D=(V(D),E(D))$ be a digraph,  the degree of a vertex $x$ in the  digraph $D$, denoted $d(x)$, is simply the degree of $x$ in
$G(D)$. Likewise, a digraph is said to be connected if
its underlying graph is connected. An {\em  oriented}  path or {\em  oriented}  cycle  is an orientation of a
path or cycle in which each vertex dominates its successor in the sequence.
An {\em  oriented} acyclic graph is a simple digraph without oriented cycles.
An {\em  oriented tree} or {\em polytree} is a  oriented acyclic graph formed by orienting the edges of undirected acyclic graphs.
A  {\em rooted tree} is an oriented tree in which all edges  are oriented  either away from or
towards the root.  Unless specifically stated, a rooted tree in this article
 is an oriented tree in which all edges  are oriented away from  the root.
An {\em oriented forest} is a disjoint union of oriented trees. A {\em rooted forest} is a disjoint union of rooted trees.

For any homogeneous ideal $I$ of the polynomial ring  $S=k[x_{1},\dots,x_{n}]$, there exists a {\em graded
minimal finite free resolution}

\vspace{3mm}
$$0\rightarrow \bigoplus\limits_{j}S(-j)^{\beta_{p,j}(I)}\rightarrow \bigoplus\limits_{j}S(-j)^{\beta_{p-1,j}(I)}\rightarrow \cdots\rightarrow \bigoplus\limits_{j}S(-j)^{\beta_{0,j}(I)}\rightarrow I\rightarrow 0,$$
where the maps are exact, $p\leq n$, and $S(-j)$ is an $S$-module obtained by shifting
the degrees of $S$ by $j$. The number
$\beta_{i,j}(I)$, the $(i,j)$-th graded Betti number of $I$, is
an invariant of $I$ that equals the number of minimal generators of degree $j$ in the
$i$th syzygy module of $I$.
Of particular interest is  the following invariant which measures the ¡°size¡± of the minimal graded
free resolution of $I$.
The regularity of $I$, denoted $\mbox{reg}\,(I)$, is defined by
$$\mbox{reg}\,(I):=\mbox{max}\,\{j-i\ |\ \beta_{i,j}(I)\neq 0\}.$$

\vspace{3mm}
Let $I$  be a monomial ideal,  $\mathcal{G}(I)$ denote the unique minimal set
of monomial generators of  $I$. We now derive some formulas for $\mbox{reg}\,(I)$ in some special cases by using some
tools developed in \cite{FHT}.

\begin{Definition} \label{bettispliting}Let $I$  be a monomial ideal, and suppose that there exist  monomial
ideals $J$ and $K$ such that $\mathcal{G}(I)$ is the disjoint union of $\mathcal{G}(J)$ and $\mathcal{G}(K)$. Then $I=J+K$
is  {\em Betti splitting} if
$$\beta_{i,j}(I)=\beta_{i,j}(J)+\beta_{i,j}(K)+\beta_{i-1,j}(J\cap K)\hspace{2mm}\mbox{for all}\hspace{2mm}i,j\geq 0,$$
where $\beta_{i-1,j}(J\cap K)=0\hspace{2mm}  \mbox{if}\hspace{2mm} i=0$.
\end{Definition}

In  \cite{FHT}, the authors describe some sufficient conditions for an
ideal $I$ to have a Betti splitting. We need  the following lemma.

\begin{Lemma}\label{lem1}(\cite[Corollary 2.7]{FHT}).
Suppose that $I=J+K$ where $\mathcal{G}(J)$ contains all
the generators of $I$ divisible by some variable $x_{i}$ and $\mathcal{G}(K)$ is a nonempty set containing
the remaining generators of $I$. If $J$ has a linear resolution, then $I=J+K$ is   Betti
splitting.
\end{Lemma}

When $I$ is a Betti
splitting ideal, Definition \ref{bettispliting} implies the following results:
\begin{Corollary} \label{cor1}
If $I=J+K$ is a Betti splitting ideal, then
$$\mbox{reg}\,(I)=\mbox{max}\,\{\mbox{reg}\,(J),\mbox{reg}\,(K),\mbox{reg}\,(J\cap K)-1\}.$$
\end{Corollary}

\medskip
The following lemmas is often used in this article.
\begin{Lemma}
\label{lem2}(\cite[Lemma 1.3]{HTT}) Let $S$ be a polynomial ring over a field and let $I$ be a proper non-zero homogeneous
ideal in $S$. Then
\[\mbox{reg}\, (I)=\mbox{reg}\,(S/I)+1.\]
\end{Lemma}

\medskip
  Let  $u\in S$ be a monomial, we set $\mbox{supp}(u)=\{x_i: x_i|u\}$. If $\mathcal{G}(I)=\{u_1,\ldots,u_m\}$, we set $\mbox{supp}(I)=\bigcup\limits_{i=1}^{m}\mbox{supp}(u_i)$. The following lemma is well known.
 \begin{Lemma}
\label{lem3}(\cite[Lemma 2.5]{HTT})
Let $S_{1}=k[x_{1},\dots,x_{m}]$ and $S_{2}=k[x_{m+1},\dots,x_{n}]$ be two polynomial rings, $I\subseteq S_{1}$ and
$J\subseteq S_{2}$ be two non-zero homogeneous  ideals. Then
\[\mbox{reg}\,(I+J)=\mbox{reg}\,(I)+\mbox{reg}\,(J)-1.\]
\end{Lemma}

 \begin{Lemma}\label{lem4}
  Let  $I, J=(u)$ be two monomial ideals  such that $\mbox{supp}\,(u)\cap \mbox{supp}\,(I)=\emptyset$. If the degree of monomial $u$  is  $d$. Then
\begin{itemize}
\item[(1)] $\mbox{reg}\,(J)=d$,
\item[(2)]$\mbox{reg}\,(JI)=\mbox{reg}\,(I)+d$.
\end{itemize}
\end{Lemma}

\medskip
\begin{Definition} \label{polarization}
Suppose that $u=x_1^{a_1}\cdots x_n^{a_n}$ is a monomial in $S$. We define the {\it polarization} of $u$ to be
the squarefree monomial $$\mathcal{P}(u)=x_{11}x_{12}\cdots x_{1a_1} x_{21}\cdots x_{2a_2}\cdots x_{n1}\cdots x_{na_n}$$
in the polynomial ring $S^{\mathcal{P}}=k[x_{ij}\mid 1\leq i\leq n, 1\leq j\leq a_i]$.
If $I\subset S$ is a monomial ideal with $\mathcal{G}(I)=\{u_1,\ldots,u_m\}$,  the  {\it polarization}
of $I$,  denoted by $I^{\mathcal{P}}$, is defined as:
$$I^{\mathcal{P}}=(\mathcal{P}(u_1),\ldots,\mathcal{P}(u_m)),$$
which is a squarefree monomial ideal in the polynomial ring $S^{\mathcal{P}}$.
\end{Definition}

\medskip
A monomial ideal $I$ and its polarization $I^{\mathcal{P}}$ share many homological and
algebraic properties.  The following is a very useful property of polarization.

\begin{Lemma}
\label{lem5}(\cite[Corollary 1.6.3]{HH2}) Let $I\subset S$ be a monomial ideal and $I^{\mathcal{P}}\subset S^{\mathcal{P}}$ its polarization.
Then
\begin{itemize}
\item[(1)] $\beta_{ij}(I)=\beta_{ij}(I^{\mathcal{P}})$ for all $i$ and $j$,
\item[(2)] $\mbox{reg}\,(I)=\mbox{reg}\,(I^{\mathcal{P}})$.
\end{itemize}
\end{Lemma}

The following lemma can be used for computing the
regularity of an ideal.
\begin{Lemma}
\label{lem6}(\cite[Lemma 1.1 and Lemma 1.2]{HTT})  Let\ \ $0\rightarrow A \rightarrow  B \rightarrow  C \rightarrow 0$\ \  be a short exact sequence of finitely generated graded $S$-modules.
Then
\begin{itemize}
\item[(1)]$\mbox{reg}\,(C)\leq \mbox{max}\, \{\mbox{reg}\,(A)-1,\mbox{reg}\,(B)\}$,
\item[(2)]$\mbox{reg}\,(B)\leq \mbox{max}\, \{\mbox{reg}\,(A),\mbox{reg}\,(C)\}$,
\item[(3)]$\mbox{reg}\,(B)= \mbox{reg}\,(A)$ if $\mbox{reg}\,(A)>\mbox{reg}\,(C)+1$,
\item[(4)]$\mbox{reg}\,(B)= \mbox{reg}\,(C)$ if $\mbox{reg}\,(C)\geq\mbox{reg}\,(A)$,
\item[(5)]$\mbox{reg}\,(C)= \mbox{reg}\,(A)-1$ if $\mbox{reg}\,(A)> \mbox{reg}\,(B)$,
\end{itemize}
\end{Lemma}

\vspace{5mm}
\section{Ordering the minimial generators of powers of edge ideals of  vertex-weighted oriented cycles}

In this section, we provide a special order on the unique minimal set of monomial generators of
 powers of edge ideals of vertex-weighted oriented cycles. Using this order, we will give some exact formulas for the
regularity of powers of edge ideals of vertex-weighted oriented cycles in  next section.

Throughout this section, let $C_n=(V(C_n),E(C_n),w)$ be  an $n$-cycle such that $w(x)\geq 2$ for any $x\in V(C_n)$ and $V(C_n)=\{x_1,\ldots,x_n\}$.  We define an order $L_1>\cdots>L_n$ on the set $\mathcal{G}(I(C_n))$ where $L_i=x_{i-1}x_i^{w_i}$ for $1\leq i\leq n$ and $x_j=x_i$ if
$j\equiv i$ mod $n$ $(1\leq i\leq n)$.
For any integer $t\geq 1$, we define an order on  the set $\mathcal{G}(I(C_n)^{t})$ as follows: We say
$M>N$ for $M,N\in \mathcal{G}(I(C_n)^{t})$ if $M=L_1^{a_1}\cdots L_n^{a_n}$, $N=L_1^{b_1}\cdots L_n^{b_n}$ such that $\sum\limits_{i=1}^{n}a_{i}=\sum\limits_{i=1}^{n}b_{i}=t$,
 we have  $({a_1},\ldots,{a_n})>_{lex}({b_1},\ldots,{b_n})$. We denoted by $L^{(t)}$  the totally ordered
set of  $\mathcal{G}(I(C_n)^{t})$ ordered in the way  above and by $L_k^{(t)}$  the $k$-th element of the set $L^{(t)}$.

\medskip
According to the order defined above, we can sort the set of generators of the following ideal.
\begin{Example} \label{exm1} Let $I(C_3)=(x_3x_1^{2},x_1x_2^2,x_2x_3^{2})$ be the edge ideal of $3$-cycle $C_3$. Then $L^{(2)}=\{(x_3x_1^{2})^2, (x_3x_1^{2})(x_1x_2^2), (x_3x_1^{2})(x_2x_3^{2}), (x_1x_2^2)^2,
(x_1x_2^2)(x_2x_3^{2}), (x_2x_3^{2})^2\}$.
\end{Example}

\medskip
We have the following fundamental fact.
\begin{Theorem}\label{thm1}
 Let $t$ be a positive integer and  $M\in \mathcal{G}(I(C_n)^{t})$, then $M$ can be shown as
 $M=L_{i_1}^{a_{i_1}}\cdots L_{i_{\ell}}^{a_{i_{\ell}}}$ with $\sum\limits_{p=1}^{\ell}a_{i_p}=t$, $a_{i_p}>0$ for $1\leq p\leq \ell$ and $1\leq i_1<\cdots < i_{\ell}\leq n$.  Moreover, the  expression of this form is unique.
\end{Theorem}
\begin{proof} It is clear that $M$ can be shown as
 $M=L_{i_1}^{a_{i_1}}\cdots L_{i_{\ell}}^{a_{i_{\ell}}}$ with $\sum\limits_{p=1}^{\ell}a_{i_p}=t$, $a_{i_p}>0$ for $1\leq p\leq \ell$ and $1\leq i_1<\cdots < i_{\ell}\leq n$.
Assume  $L_{i_1}^{a_{i_1}}\cdots L_{i_{\ell}}^{a_{i_{\ell}}}$ and $L_{j_1}^{b_{j_1}}\cdots L_{j_m}^{b_{j_m}}$ are two expressions of $M$ with $\sum\limits_{p=1}^{\ell}a_{i_p}=\sum\limits_{q=1}^{m}b_{i_q}=t$,
 where $a_{i_p}, b_{j_q}>0$ for any $1\leq p\leq \ell$, $1\leq q\leq m$ and  $1\leq i_1<\cdots < i_{\ell}\leq n$, $1\leq j_1<\cdots < j_m\leq n$. We will show that $\ell=m$, $i_p=j_p$ and $a_{i_p}=b_{j_p}$ for $1\leq p\leq \ell$.
 We use induction on $t$. Case $t=1$ is clear. Now we assume $t\geq 2$.

  Claim: $\{i_1,\ldots,i_{\ell}\}\cap \{j_1,\ldots,j_m\}\neq\emptyset$, thus
  we  assume that $i_1=j_1$. It follows that
 $$L_{i_1}^{a_{i_1}-1}\cdots L_{i_{\ell}}^{a_{i_{\ell}}}=L_{j_1}^{b_{j_1}-1}\cdots L_{j_m}^{b_{j_m}}$$
 Therefore, by  induction hypothesis, we obtain that
  $\ell=m$, $i_p=j_p$ and $a_{i_p}=b_{j_p}$ for $1\leq p\leq \ell$, as desired.

 In fact, if  $\{i_1,\ldots,i_{\ell}\}\cap \{j_1,\ldots,j_m\}=\emptyset$.  For  any  $1\leq p\leq \ell$, $L_{i_p}$ is a factor of monomial $L_{i_1}^{a_{i_1}}\cdots L_{i_{\ell}}^{a_{i_{\ell}}}$, thus it is also
a factor of monomial $L_{j_1}^{b_{j_1}}\cdots L_{j_m}^{b_{j_m}}$. Hence  there exists $1\leq s\leq m$ such that $L_{j_s}=x_{i_p}x_{i_p+1}^{w_{i_p+1}}$.  By the expression of  $L_{i_1}^{a_{i_1}}\cdots L_{i_{\ell}}^{a_{i_{\ell}}}$,
 we obtain
$$L_{j_1}^{b_{j_1}}\cdots L_{j_m}^{b_{j_m}}=(x_{i_1}x_{i_1+1}^{w_{i_{1}+1}})^{a_{i_1}w_{i_1}}\cdots (x_{i_{\ell}}x_{i_{\ell}+1}^{w_{i_{\ell}+1}})^{a_{i_{\ell}}w_{i_{\ell}}}M'$$
where $M'$ is a monomial. By comparing the degree of monomials $L_{i_1}^{a_{i_1}}\cdots L_{i_{\ell}}^{a_{i_{\ell}}}$ and $L_{j_1}^{b_{j_1}}\cdots L_{j_m}^{b_{j_m}}$, we get
 $$\sum\limits_{p=1}^{\ell}a_{i_p}(1+w_{i_p})\geq \sum\limits_{p=1}^{\ell}a_{i_p}w_{i_{p}}(1+w_{i_p+1}).$$
This implies $\sum\limits_{p=1}^{\ell}a_{i_p}(1-w_{i_p}w_{i_p+1})\geq0$, a contradiction.
\end{proof}

\medskip
\begin{Definition}\label{def1}
Let $1\leq k<t$ be two  integers, and  $M_1\in \mathcal{G}(I(C_n)^{k})$, $M_2\in \mathcal{G}(I(C_n)^{t})$.
We denoted by $M_1\mid^{edge}M_2$  if there exists $M_3\in \mathcal{G}(I(C_n)^{t-k})$ such that $M_2=M_1M_3$. Otherwise, we denoted by $M_1\nmid\,^{edge}M_2$.
\end{Definition}

\medskip
The following three results are needed.
\begin{Lemma}\label{lem7}
 Let $L_{i}^{(2)}, L_{j}^{(2)}\in L^{(2)}$ such that $L_{i}^{(2)}>L_{j}^{(2)}$,
then there exists $L_{k}^{(2)}\in L^{(2)}$  such that $L_{i}^{(2)}>L_{k}^{(2)}$ and
 $(L_{k}^{(2)}:L_{i}^{(2)})$ has one of the following two forms:
\begin{enumerate}
\item[(1)] $(L_{k}^{(2)}:L_{i}^{(2)})=(L_{\ell_2}:L_{\ell_1})$, where  $L_{\ell_1}>L_{\ell_2}$, $L_{\ell_2}\mid^{edge}L_{k}^{(2)}$ and $L_{\ell_1}\mid^{edge}L_{i}^{(2)}$;
\item[(2)] $(L_{k}^{(2)}:L_{i}^{(2)})=(L_{n-2}L_{n}:L_{1}L_{n-1})$, where $L_{n}\mid^{edge}L_{k}^{(2)}$,  $L_{n-2}\mid^{edge}L_{k}^{(2)}$, $L_{n-1}\mid^{edge}L_{i}^{(2)}$ and $L_{1}\mid^{edge}L_{i}^{(2)}$.
\end{enumerate}
Furthermore, $(L_{j}^{(2)}:L_{i}^{(2)})\subseteq (L_{k}^{(2)}:L_{i}^{(2)})$.
 \end{Lemma}
\begin{proof}
 Set  $L_{j}^{(2)}=L_{j_1}L_{j_2}$ with $1\leq j_1\leq j_2\leq n$. If  there exists some $1\leq a\leq 2$ such that $L_{j_a}\mid^{edge}L_{i}^{(2)}$. For convenience, we  assume $a=1$, thus we get $(L_{j}^{(2)}:L_{i}^{(2)})=(L_{j_2}^{(1)}:L_{i'}^{(1)})$ and $L_{i'}^{(1)}>L_{j_2}^{(1)}$,
where  $L_{i'}^{(1)}=\frac{L_{i}^{(2)}}{L_{j_1}}$.
Choose  $k=j$, $\ell_2=j_2$ and $\ell_1=i'$, the result holds.

Otherwise, if $L_{j_a}\nmid^{edge}L_{i}^{(2)}$ for any  $1\leq a\leq 2$, then  $j_1\geq2$. In this case,
 we consider the following two cases:

 (i) If there exists some $1\leq r\leq 2$ such that $(L_{j}^{(2)}:L_{i}^{(2)})\subseteq (x_{j_r}^{w_{j_r}})$. Let $L_{i}^{(2)}=L_{i_1}L_{i_2}$ with $1\leq i_1\leq i_2\leq n$. Choose $b=2$ if $L_{i_2}>L_{j_r}$, otherwise $b=1$,  and  $L_{k}^{(2)}=\frac{L_{i}^{(2)}}{L_{i_b}}L_{j_r}$, thus we get $L_{i_b+1}\geq L_{j_r}$, $L_{i}^{(2)}>L_{k}^{(2)}$ and  $(L_{k}^{(2)}:L_{i}^{(2)})=(L_{j_r}:L_{i_b})$. Notice that
 \[
 (L_{j}^{(2)}:L_{i}^{(2)})=(\frac{L_{j}^{(2)}}{\gcd(L_{j}^{(2)},L_{i}^{(2)})})\ \ \ \text{and}\ \ \
 (L_{j_r}:L_{i_b})=\left\{\begin{array}{ll}
(x_{j_r}^{w_{j_r}})\ \ &\text{if}\  \  L_{j_r}=L_{i_b+1},\\
(L_{j_r})\ \ &
\text{if}\ \ L_{i_b+1}>L_{j_r}.
\end{array}\right.
\]
  If  $L_{j_r}=L_{i_b+1}$, then the result is true. Otherwise,  it is enough to show that $x_{j_{r}-1}$  is not a factor of $\gcd(L_{j}^{(2)},L_{i}^{(2)})$. Thus we obtain $L_{j_r}$ is a factor of  generator $\frac{L_{j}^{(2)}}{\gcd(L_{j}^{(2)},L_{i}^{(2)})}$ of $(L_{j}^{(2)}:L_{i}^{(2)})$, the assertion follows   from the formula  above. In fact, if
  $x_{j_{r}-1}$  is a factor of $L_{i}^{(2)}$. By  the expression of $L_j^{(2)}$ and  the  hypothesis that  $L_{j_a}\nmid\,^{edge}L_{i}^{(2)}$ for any $a=1,2$,  we  obtain  $L_{j_r-1}\mid^{edge}L_{i}^{(2)}$.
 It follows  $L_{j_r-1}=L_{i_b}$ by the definition of $b$, contradicting with  the  hypothesis $L_{i_b+1}>L_{j_r}$.

(ii) If $(L_{j}^{(2)}:L_{i}^{(2)})\nsubseteq (x_{j_r}^{w_{j_r}})$ for any $1\leq r\leq 2$, then  $x_{j_r}$  is a factor of $\gcd(L_{j}^{(2)},L_{i}^{(2)})$
from the expression of $L_j^{(2)}$ and the formula of $(L_{j}^{(2)}:L_{i}^{(2)})$.
 This implies
$L_{j_r+1} \mid^{edge}L_{i}^{(2)}$ by  the  hypotheses that  $L_{j_a}\nmid\,^{edge}L_{i}^{(2)}$ for any $a=1,2$.
Thus $L_{i}^{(2)}$ has the form
$$L_{i}^{(2)}=L_{j_1+1}L_{j_2+1}.$$
It follows $j_2=n$  by the expression of $L_j^{(2)}$ and $L_{i}^{(2)}>L_{j}^{(2)}$.
Claim:   $j_1\neq n-1,n$. In fact, if $j_1=n-1$, then $L_{n} \mid^{edge}L_{i}^{(2)}$ and $L_{n} \mid^{edge}L_{j}^{(2)}$, contradicting with  the  hypothesis  $L_{j_a}\nmid\,^{edge}L_{i}^{(2)}$ for any $a=1,2$.
If  $j_1=n$, then $(L_{j}^{(2)}:L_{i}^{(2)})\subseteq (x_{n}^{w_{n}})$, contradicting with $(L_{j}^{(2)}:L_{i}^{(2)})\nsubseteq (x_{j_r}^{w_{j_r}})$ for any $1\leq r\leq 2$.
 Hence $j_1\leq n-2$, which implies $n\geq 4$ because of $j_1\geq 2$. We consider the following two cases:

 (i) If  $2\leq j_1< n-2$, then $n\geq 5$ and $j_1+1< n-1$. Choose $L_{k}^{(2)}=L_{j_1+1}L_{n}$, we obtain
 $(L_{k}^{(2)}:L_{i}^{(2)})=(L_n:L_{1})=(x_{n-1}x_n^{w_n-1})$, $L_{i}^{(2)}>L_{k}^{(2)}$ and $(L_{j}^{(2)}:L_{i}^{(2)})\subseteq(x_{n-1}x_n^{w_n-1})$, which implies $(L_{j}^{(2)}:L_{i}^{(2)})\subseteq (L_{k}^{(2)}:L_{i}^{(2)})$.

 (ii)  If  $j_1=n-2$, then we choose  $k=j$. Thus $L_{i}^{(2)}>L_{k}^{(2)}$ and $(L_{j}^{(2)}:L_{i}^{(2)})=(L_{k}^{(2)}:L_{i}^{(2)})=(L_{n-2}L_n:L_{n-1}L_{1})$.
\end{proof}

\medskip
The next two theorems are the most important technical results of this section. They play  vital roles in calculating the regularity of  powers of edge ideals of vertex-weighted oriented cycles in the next section.
\begin{Theorem}\label{thm2}
 Let $t$ be a positive integer, $\ell=\mbox{min}\,\{t,\lfloor\frac{n}{2}\rfloor\}-1$, where $\lfloor\frac{n}{2}\rfloor$ denotes  the largest integer  $\leq \frac{n}{2}$, and let $L_{i}^{(t)}, L_{j}^{(t)}\in L^{(t)}$ with $L_{i}^{(t)}>L_{j}^{(t)}$,
then there exists some $L_{k}^{(t)}\in L^{(t)}$  such that $L_{i}^{(t)}>L_{k}^{(t)}$ and
 $(L_{k}^{(t)}:L_{i}^{(t)})$ has one of the following two forms:
\begin{enumerate}
\item[(1)]   $(L_{k}^{(t)}:L_{i}^{(t)})=(L_{\ell_2}:L_{\ell_1})$, where $L_{\ell_1}>L_{\ell_2}$,  $L_{\ell_2}\mid^{edge}L_{k}^{(t)}$  and $L_{\ell_1}\mid^{edge}L_{i}^{(t)}$;
\item[(2)] $(L_{k}^{(t)}:L_{i}^{(t)})=(\prod\limits_{s=0}^{q}L_{n-2s}:\prod\limits_{s=0}^{q}L_{n+1-2s})$ for some $q\leq \ell$, where $L_{n-2s}\mid^{edge}L_{k}^{(t)}$,  $L_{n+1-2s}\mid^{edge}L_{i}^{(t)}$  for any $0\leq s \leq q$, and
$n+1-2s\equiv j$   mod $n$ for some $0<j\leq n$.
\end{enumerate}
 Furthermore, $(L_{j}^{(t)}:L_{i}^{(t)})\subseteq (L_{k}^{(t)}:L_{i}^{(t)})$.
 \end{Theorem}
\begin{proof}
We proceed by induction on $t$. Case $t=1$ holds if we choose $k=\ell_2=j$, $\ell_1=i$.  Case $t=2$ holds from Lemma
\ref{lem7}.  Now suppose that $t\geq 3$.
Set $L_{j}^{(t)}=L_{j_1}\cdots L_{j_t}$ with  $1\leq j_1\leq\cdots\leq j_t\leq n$. Similar  to Lemma \ref{lem7}, we consider the following two cases:

(I) If there exists some $1\leq a\leq t$ such that $L_{j_a}\mid^{edge}L_{i}^{(t)}$, then
\[
(L_{j}^{(t)}:L_{i}^{(t)})=(L_{j'}^{(t-1)}:L_{i'}^{(t-1)})
\]
where $L_{j'}^{(t-1)}=\frac{L_{j}^{(t)}}{L_{j_a}}$ and  $L_{i'}^{(t-1)}=\frac{L_{i}^{(t)}}{L_{j_a}}$.
By induction hypothesis, there exists some  $k'$   such that $L_{i'}^{(t-1)}>L_{k'}^{(t-1)}$ and
$(L_{k'}^{(t-1)}:L_{i'}^{(t-1)})$ is one of the following two forms:

(i) $(L_{k'}^{(t-1)}:L_{i'}^{(t-1)})=(L_{\ell_2}:L_{\ell_1})$ with $L_{\ell_1}>L_{\ell_2}$,  $L_{\ell_2}\mid^{edge}L_{k'}^{(t-1)}$  and $L_{\ell_1}\mid^{edge}L_{i'}^{(t-1)}$.

(ii)  $(L_{k'}^{(t-1)}:L_{i'}^{(t-1)})=(\prod\limits_{s=0}^{q'}L_{n-2s}:\prod\limits_{s=0}^{q'}L_{n+1-2s})$  for some $q'\leq \ell'$,
 where $\ell'=\mbox{min}\,\{t-1,\lfloor\frac{n}{2}\rfloor\}-1$, $L_{n-2s}\mid^{edge}L_{k'}^{(t-1)}$,  $L_{n+1-2s}\mid^{edge}L_{i'}^{(t-1)}$  for any $0\leq s \leq q'$ and  $n+1-2s\equiv j$  mod $n$ for some $0<j\leq n$.

We choose $L_{k}^{(t)}=L_{j_a}L_{k'}^{(t-1)}$, then  $L_{i}^{(t)}>L_{k}^{(t)}$ and  $(L_{k}^{(t)}:L_{i}^{(t)})= (L_{k'}^{(t-1)}:L_{i'}^{(t-1)})$, as desired.

(II)  If $L_{j_a}\nmid\,^{edge}L_{i}^{(t)}$ for any $1\leq a\leq t$, then  $j_1\geq 2$ because of $L_{i}^{(t)}>L_{j}^{(t)}$.
 We consider the following two cases:

(i) If there exists some  $1\leq r\leq t$ such that $(L_{j}^{(t)}:L_{i}^{(t)})\subseteq (x_{j_r}^{w_{j_r}})$. Set $L_{i}^{(t)}=L_{i_1}\cdots L_{i_t}$ with $1\leq i_1\leq\cdots\leq i_t\leq n$.
We choose $L_{i_b}=\mbox\,{min}\{L_{i_b}\mid L_{i_b}>L_{j_r}\ \mbox{and}\  L_{i_b}\mid^{edge} L_{i}^{(t)} \}$
 and $L_{k}^{(t)}=\frac{L_{i}^{(t)}}{L_{i_b}}L_{j_r}$, thus $L_{i}^{(t)}>L_{k}^{(t)}$ and $(L_{k}^{(t)}:L_{i}^{(t)})=(L_{j_r}:L_{i_b})$.
Similar arguments as  Lemma \ref{lem7},
we have
$(L_{j}^{(t)}:L_{i}^{(t)})\subseteq (L_{j_r}:L_{i_b})$.
 Hence the conclusion holds.

(ii) If $(L_{j}^{(t)}:L_{i}^{(t)})\nsubseteq (x_{j_r}^{w_{j_r}})$ for any $1\leq r\leq t$, then  $n\geq 4$ because of $L_{j_a}\nmid\,^{edge}L_{i}^{(t)}$ for any $1\leq a\leq t$, which implies $\ell\geq 1$. Since $(L_{j}^{(t)}:L_{i}^{(t)})=(\frac{L_{j}^{(t)}}{\gcd(L_{j}^{(t)},L_{i}^{(t)})})$, we have  $x_{j_r}$  is a factor of  $L_i^{(t)}$.
 This implies
$L_{j_r+1} \mid^{edge}L_{i}^{(t)}$ for $1\leq r\leq t$ by  the  hypothesis  that $L_{j_a}\nmid\,^{edge}L_{i}^{(t)}$ for any $1\leq a\leq t$.
Thus  $L_{i}^{(t)}$ has the form
$$L_{i}^{(t)}=L_{j_1+1}\cdots L_{j_t+1}.$$
This implies $L_{j_t+1}=L_1$  by the expression of $L_j^{(t)}$ and $L_{i}^{(t)}>L_{j}^{(t)}$.
It follows that $L_{j_t}=L_n$, i.e., $j_t=n$.

 Choose $q=\mbox{max}\{q\mid L_{j_{t-q}}=L_{n-2q}, L_{n-2q}\mid^{edge} L_{j}^{(t)}\ \mbox{for any}\ \  0\leq q\leq \ell\}$, thus   $L_{n-2s}\mid^{edge}L_{j}^{(t)}$ and   $L_{n+1-2s}\mid^{edge}L_{i}^{(t)}$  for any $0\leq s \leq q$. Next,  we consider the following two cases:

(i) If $q=t-1$, then $\ell=t-1$. In this case, $L_j^{(t)}=\prod\limits_{s=0}^{t-1}L_{n-2s}$ and  $L_i^{(t)}=\prod\limits_{s=0}^{t-1}L_{n-2s+1}$. Choose
$k=j$, as desired.

(ii) If $q\leq t-2$, then $n\geq 2\ell+2 \geq 2q+2$ by the definition of $q$ and $\ell$, $L_j^{(t)}=Q_1\prod\limits_{s=0}^{q}L_{n-2s}$ and  $L_i^{(t)}=Q_2\prod\limits_{s=0}^{q}L_{n-2s+1}$, where $Q_1=\frac{L_j^{(t)}}{\prod\limits_{s=0}^{q}L_{n-2s}}=L_{j_1}\cdots L_{j_{t-q-1}}$ and $Q_2=\frac{L_i^{(t)}}{\prod\limits_{s=0}^{q}L_{n+1-2s}}=L_{j_1+1}\cdots L_{j_{t-q-1}+1}$.  Choose $L_k^{(t)}=Q_2\prod\limits_{s=0}^{q}L_{n-2s}$, then $L_{i}^{(t)}>L_{k}^{(t)}$, $(L_{k}^{(t)}:L_{i}^{(t)})=(\prod\limits_{s=0}^{q}L_{n-2s}:\prod\limits_{s=0}^{q}L_{n+1-2s})$ and  $L_{n-2s}\mid^{edge}L_{k}^{(t)}$,  $L_{n+1-2s}\mid^{edge}L_{i}^{(t)}$  for any $0\leq s \leq q$.

Next we  prove $(L_{j}^{(t)}:L_{i}^{(t)})\subseteq (L_{k}^{(t)}:L_{i}^{(t)})$.
This is equivalent to prove $uL_{i}^{(t)}\in (L_{k}^{(t)})$ for  any $u\in (L_{j}^{(t)}:L_{i}^{(t)})$. It  is enough  to prove $\prod\limits_{s=0}^{q}L_{n-2s}\mid u\prod\limits_{s=0}^{q}L_{n-2s+1}$ by the expression of $L_{i}^{(t)}$ and $L_{k}^{(t)}$. In fact, let $u\in (L_{j}^{(t)}:L_{i}^{(t)})$, then $uL_{i}^{(t)}\in (L_{j}^{(t)})$. This implies  $Q_1\prod\limits_{s=0}^{q}L_{n-2s}\mid u\,Q_2\prod\limits_{s=0}^{q}L_{n+1-2s}$ because of   $L_j^{(t)}=Q_1\prod\limits_{s=0}^{q}L_{n-2s}$ and  $L_i^{(t)}=Q_2\prod\limits_{s=0}^{q}L_{n-2s+1}$.
It is  sufficient  to show $\supp\,(\prod\limits_{s=0}^{q}L_{n-2s})\cap \supp\,(Q_2)=\emptyset$, it is equivalent to  $L_{j_{t-q-1}+1}\notin \{L_{n-2q+1},L_{n-2q},L_{n-2q-1}\}$. In fact, if $L_{j_{t-q-1}+1}= L_{n-2q+1}$, then $L_{j_{t-q-1}}= L_{n-2q}$ and $(L_{j}^{(t)}:L_{i}^{(t)})\subseteq (x_{n-2q}^{w_{n-2q}})$
 because of $L_{n-2q}^2\mid^{edge}L_j^{(t)}$, the expression of $L_{i}^{(t)}$ and $w_{n-2q}\geq 2$,  contradicting  with the hypothesis  $(L_{j}^{(t)}:L_{i}^{(t)})\nsubseteq (x_{j_r}^{w_{j_r}})$ for any $1\leq r\leq t$.
If  $L_{j_{t-q-1}+1}=L_{n-2q}$, then  $L_{i}^{(t)}$ and  $L_{j}^{(t)}$ have a common edge $L_{n-2q}$, a contradiction.
If  $L_{j_{t-q-1}+1}=L_{n-2q-1}$, then $L_{j_{t-q-1}}= L_{n-2q-2}$,  contradicting  with the choice of $q$.
This proof is complete.
\end{proof}

\medskip
\begin{Theorem}\label{thm3}  Let $t$ be  a positive integer, $\ell=\mbox{min}\,\{t,\lfloor\frac{n}{2}\rfloor\}-1$, $L^{(t)}=\{L_1^{(t)}, \ldots,L_r^{(t)}\}$ a totally ordered set of all elements of
  $\mathcal{G}(I(C_n)^{t})$ such that $L_1^{(t)}> \cdots >L_r^{(t)}$. For any $1\leq i\leq r$, we write $L_i^{(t)}$ as
$L_i^{(t)}=L_{i_1}^{a_{i_1}}\cdots L_{i_{k_i}}^{a_{i_{k_i}}}$ with $1\leq i_1<\cdots < i_{k_i}\leq n$,
$\sum\limits_{j=1}^{{k_i}}a_{i_j}=t$ and $a_{i_j}>0$ for $j=1,\ldots,k_i$. For $1\leq i\leq r-1$, let
$J_i=(L_{i+1}^{(t)}, \ldots, L_r^{(t)})$, $K_i=((L_{i_1+1},\ldots, L_n):L_{i_1})+\sum\limits_{j=1}^{p_i}(L_{i_j+1}:L_{i_j})$, where  if $i_{k_i}=n$, then
${p_i}=k_i-1$, otherwise, ${p_i}=k_i$.
\begin{enumerate}
\item[(1)] If $i_1=1$, then  $(J_i:L_{i}^{(t)})=K_i+Q_i$
and  $Q_{i}=\sum\limits_{j=0}^{q_i}(\prod\limits_{s=0}^{j}L_{n-2s}:\prod\limits_{s=0}^{j}L_{n+1-2s})$, where $q_i=\mbox\,{max}\{q:L_{n+1-2q}|^{edge}L_{i}^{(t)}\ \mbox{for any}\  0\leq q\leq \ell\}$;
\item[(2)] If $i_1\geq2$,  then $(J_i:L_{i}^{(t)})=K_i$.
\end{enumerate}
\end{Theorem}
\begin{proof} It is obvious for $t=1$. Now  assume that $t\ge 2$.
Set  $M_{j}=\frac{L_{i}^{(t)}}{L_{i_{1}}}L_{j}$ for any $i_1+1\leq j\leq n$, $N_{j}=\frac{L_{i}^{(t)}}{L_{i_{j}}}L_{i_{j}+1}$ for any $1\leq j\leq p_i$, then $(M_{i_1+1},\ldots,M_{n},N_{1},\ldots,N_{p_i})\subseteq J_{i}$.
Hence
\begin{eqnarray*}
K_i&=&((L_{i_1+1},\ldots, L_n):L_{i_1})+\sum\limits_{j=1}^{p_i}(L_{i_j+1}:L_{i_j})\\
&=&((M_{i_1+1},\ldots,M_{n}):L_{i}^{(t)})+((N_{1},\ldots,N_{p_i}):L_{i}^{(t)})\\
&=&((M_{i_1+1},\ldots,M_{n},N_{1},\ldots,N_{p_i}):L_{i}^{(t)})\subseteq (J_{i}:L_i^{(t)}).
\end{eqnarray*}

We distinguish into the following two cases:

(i) If $i_1\geq 2$, then  $L_1\nmid^{\,edge}L_{i}^{(t)}$. For any monomial  $u\in \mathcal {G}(J_{i}:L_{i}^{(t)})$, then by  Theorem \ref{thm2}, there exists $L_{\ell_{1}}$, $L_{\ell_2}$, $L^{(t)}_{a}\in J_i$ for some $i+1\leq a\leq r$ such that $u\in (L_{\ell_2}:L_{\ell_1})$, $L_{\ell_{1}}>L_{\ell_2}$,  $L_{\ell_{2}}\mid^{edge}L_{a}^{(t)}$  and $L_{\ell_{1}}\mid^{edge}L_{i}^{(t)}$,
 which implies  $\ell_{2}>\ell_{1}\geq i_1$. Hence
$(L_{\ell_{2}}:L_{\ell_{1}})\subseteq K_i$, as desired.

(ii) If $i_1=1$, then $L_{1}\mid^{edge} L_{i}^{(t)}$.  By the definition of $q_i$, we get $\prod\limits_{s=0}^{j}L_{n+1-2s}\mid L_{i}^{(t)}$ for any $0\leq j\leq q_i$.
Set $T_{j}=\frac{L_{i}^{(t)}}{\prod\limits_{s=0}^{j}L_{n+1-2s}}\prod\limits_{s=0}^{j}L_{n-2s}$, we obtain  $L_{i}^{(t)}>T_{j}$.
It follows that $T_{j}\in J_{i}$. Hence
 \[
Q_i=\sum\limits_{j=0}^{q_i}(\prod\limits_{s=0}^{j}L_{n-2s}:\prod\limits_{s=0}^{j}L_{n+1-2s})=\sum\limits_{j=0}^{q_i}(T_{j}:L_{i}^{(t)})
=((T_{0},\ldots,T_{q_i}):L_{i}^{(t)})\subseteq(J_{i}:L_{i}^{(t)}).
\]
On the other hand,
$(J_{i}:L_{i}^{(t)})=((L_{i+1}^{(t)}, \ldots, L_r^{(t)}):L_{i}^{(t)})=\sum\limits_{j=i+1}^{r}(L_j^{(t)}:L_{i}^{(t)})$.
If there exists some $q_i\in\{1,\dots,\ell\}$, then $(J_{i}:L_{i}^{(t)})\subseteq K_i+Q_i$. Otherwise, $(J_{i}:L_{i}^{(t)})\subseteq K_i$.
We complete the  proof.
\end{proof}

\medskip
\section{Regularity of  powers of edge ideals of vertex-weighted oriented cycles}

In this section,  we give exact formulas for the
regularity of powers of edge ideals of vertex-weighted oriented cycles.  Meanwhile, we also give some examples to show  the  regularity of powers of edge ideals of vertex-weighted oriented cycles is related to direction selection and  the assumption  that $w(x)\geq 2$ for any vertex $x$  cannot be dropped.

A hypergraph $H=(X,\mathcal {E})$ over the vertex set $X=\{x_1, \ldots, x_n\}$
consists of $X$ and a collection $\mathcal {E}$ of nonempty subsets of $X$, these subsets are called the
edges of $H$. Let $Y\subseteq X$, the induced subhypergraph of $H$ on $Y$,
denoted by $H[Y]$, is the hypergraph with the vertex set $Y$ and the edge set $\{E\in\mathcal {E}\mid E\subseteq Y\}$. A hypergraph $H$ is simple if there is no containment between any pair of its edges.

\medskip
We need the following two lemmas.
\begin{Lemma}\label{lem8}(\cite[Lemma 3.1]{H})  Let $H$ be a simple hypergraph. Then $\mbox{reg}\,(H')\leq \mbox{reg}\,(H)$
 for any induced subhypergraph $H'$ of $H$.
\end{Lemma}

\begin{Lemma}\label{lem9}(\cite[Proposition 4.1]{LM})
Let $I\subseteq S=k[x_1,\ldots,x_n]$ be a squarefree monomial ideal satisfying every minimal generator of $I$ contains at least one variable not dividing any
other generator of $I$. Then
$$reg\,(I)=|X|-|\mathcal {G}(I)|+1 $$
where $X=supp(I)$.
\end{Lemma}

 For convenience,  all of notations used in the following two propositions and Theorem \ref{thm4}  are as those of Theorem  \ref{thm3}.
\begin{Proposition}\label{prop1}
Let $L^{(t)}$,  $L_i^{(t)}$,$J_i$, $K_i$ and $Q_i$ be as Theorem \ref{thm3}. For any $1\leq i\leq r-1$,
\begin{enumerate}
\item[(1)]  If $i_1=1$ and $q_i=0$, then $\mbox{reg}\,((J_i:L_{i}^{(t)}))=\sum\limits_{j=2}^{n}w_j-n+1$;
\item[(2)] If $i_1\geq 2$, then $\mbox{reg}\,((J_i:L_{i}^{(t)}))=\sum\limits_{j=i_1+1}^{n}w_j-(n-i_1)+1$.
\end{enumerate}
\end{Proposition}
\begin{proof}
(1) If  $q_i=0$, then $L_{n-1}\nmid^{\,edge}L_{i}^{(t)}$. If $i_1=1$, then $i_{p_i}<n-1$ by the definition of $p_i$. Thus
\begin{eqnarray*}
K_i&=&((L_{2},\ldots, L_n):L_{1})+\sum\limits_{j=1}^{p_i}(L_{i_j+1}:L_{i_j})\\
&=&(x_{2}^{w_{2}},x_{2}x_{3}^{w_{3}},\ldots,x_{n-1}x_n^{w_n-1})+\sum\limits_{j=1}^{p_i}(x_{i_j+1}^{w_{i_j+1}}),\\
(J_{i}:L_{i}^{(t)})&=&K_i+Q_i=K_i.
 \end{eqnarray*}
Let $K_i^{\mathcal {P}}$ be the polarization of the ideal $K_i$, then  $|\supp\,(K_i^{\mathcal {P}})|=\sum\limits_{j=2}^{n}w_j-1$ and $|\mathcal {G}(K_i^{\mathcal {P}})|=n-1$.
Notice a fact that  $x_{j,w_j}$ is only a factor of the unique monomial $x_{j-1,1}\prod\limits_{k=1}^{w_{j}}x_{j,k}$ or  $\prod\limits_{k=1}^{w_{j}}x_{j,k}$   of  the set $\mathcal {G}(K_i^{\mathcal {P}})$ for any $2\leq j\leq n-1$ and $x_{n,w_n-1}$ is also only a factor of the unique monomial $x_{n-1,1}\prod\limits_{j=1}^{w_{n}-1}x_{n,j}$ of  the set $\mathcal {G}(K_i^{\mathcal {P}})$.
Hence  by Lemma \ref{lem5} (2) and Lemma \ref{lem9}, we obtain
\begin{eqnarray*}
\mbox{reg}\,((J_{i}:L_{i}^{(t)}))&=&\mbox{reg}\,(K_i)=\mbox{reg}\,(K_i^{\mathcal {P}})=|supp(K_i^{\mathcal {P}})|-|\mathcal {G}(K_i^{\mathcal {P}})|+1\\
&=&(\sum\limits_{j=2}^{n}w_j-1)-(n-1)+1=\sum\limits_{j=2}^{n}w_j-n+1.
\end{eqnarray*}

(2) If $i_1\geq2$, then by Theorem \ref{thm3} (2), we obtain
\begin{eqnarray*}
(J_{i}:L_{i}^{(t)})&=&K_i=((L_{i_1+1},\ldots, L_n):L_{i_1})+\sum\limits_{j=1}^{p_i}(L_{i_j+1}:L_{i_j})\\
&=&(x_{i_1+1}^{w_{i_1+1}},x_{i_1+1}x_{i_1+2}^{w_{i_1+2}},\ldots,x_{n-1}x_n^{w_n})+\sum\limits_{j=1}^{p_i}(x_{i_j+1}^{w_{i_j+1}})
\end{eqnarray*}
Let $K_i^{\mathcal {P}}$ be the polarization of the ideal $K_i^{\mathcal {P}}$, then  $|\supp\,(K_i^{\mathcal {P}})|=\sum\limits_{j=i_1+1}^{n}w_j$ and $|\mathcal {G}(K_i^{\mathcal {P}})|=n-i_1$. Similar arguments as the proof of (1), we get
$$
\mbox{reg}\,((J_{i}:L_{i}^{(t)}))=\sum\limits_{j=i_1+1}^{n}w_j-(n-i_1)+1.
$$
\end{proof}

\begin{Proposition}\label{prop2}
Let $L^{(t)}$,  $L_i^{(t)}$,$J_i$, $K_i$ and $Q_i$ be as Theorem \ref{thm3}.   For any $1\leq i\leq r-1$. If $i_1=1$ and $q_i\geq1$, then $$\mbox{reg}\,((J_i:L_{i}^{(t)}))\leq \sum\limits_{j=2}^{n}w_j-n+1.$$
\end{Proposition}
\begin{proof} Since $i_1=1$ and $q_i\geq1$, we have  $L_{j}\mid^{edge}L_{i}^{(t)}$ for $j=1,n-1$. It follows that  $i_{p_i}=n-1$.
Thus
\begin{eqnarray*}
\hspace{2.0cm} K_i&=&(x_{2}^{w_{2}},x_{2}x_{3}^{w_{3}},\ldots,x_{n-1}x_n^{w_n-1})+\sum\limits_{j=1}^{p_i}(x_{i_j+1}^{w_{i_j+1}}),\hspace{3.0cm}(1)\\
\hspace{2.0cm} Q_i&=&\sum\limits_{j=0}^{q_i}(\prod\limits_{s=0}^{j}L_{n-2s}:\prod\limits_{s=0}^{j}L_{n+1-2s})=(u_0,u_1,\ldots, u_{q_i}), \hspace{2.3cm}(2)
\end{eqnarray*}
where monomial $u_j=\frac{\prod\limits_{s=0}^{j}L_{n-2s}}{gcd(\prod\limits_{s=0}^{j}L_{n-2s},\prod\limits_{s=0}^{j}L_{n+1-2s})}$ for   $0\leq j\leq q_i$.

Let $$T_{j}=K_i+(u_0,u_{1}\ldots,u_{j})\ \  \text{for any}\ \ 0\leq j\leq q_i,\eqno(3)$$
then $(J_{i}:L_{i}^{(t)})=K_i+Q_i=T_{q_i}$.

For  $0\leq j\leq q_i$, we will prove
$$\mbox{reg}\,(T_{j})\leq \sum\limits_{j=2}^{n}w_j-n+1, \eqno(4)$$
thus the result follows.

Now we prove  formulas  (4) by induction on $j$.

If  $j=0$, then
$$T_{0}=K_i+(u_0)=K'_i+(x_n^{w_n}),$$
 where $K'_i=(x_{2}^{w_{2}},x_{2}x_{3}^{w_{3}},\ldots,x_{n-1}x_n^{w_n-1})+\sum\limits_{j=1}^{p_i-1}(x_{i_j+1}^{w_{i_j+1}})$.

Let ${K'_i}^{\mathcal {P}}$ be the polarization of the ideal $K'_i$, then  $|\supp\,({K'_i}^{\mathcal {P}})|=\sum\limits_{j=2}^{n}w_j-1$ and $|\mathcal {G}({K'_i}^{\mathcal {P}})|=n-1$.
Since  $x_{j,w_j}$ is only a factor of the unique monomial $x_{j-1,1}\prod\limits_{k=1}^{w_{j}}x_{j,k}$ or  $\prod\limits_{k=1}^{w_{j}}x_{j,k}$ of  the set $\mathcal {G}({K'_i}^{\mathcal {P}})$ for any $2\leq j\leq n-1$ and $x_{n,w_n-1}$ is also only a factor of the unique monomial $x_{n-1,1}\prod\limits_{j=1}^{w_{n}-1}x_{n,j}$ of  the set $\mathcal {G}({K'_i}^{\mathcal {P}})$,
 we obtain  by Lemma \ref{lem5} (2) and Lemma \ref{lem9},
\[
\mbox{reg}\,(K'_i)=\mbox{reg}\,({K'_i}^{\mathcal {P}})=|\supp\,({K'_i}^{\mathcal {P}})|-|\mathcal {G}({K'_i}^{\mathcal {P}})|+1\\
=\sum\limits_{j=2}^{n}w_j-n+1. \eqno (5)
\]
Notice that $(K'_i:x_n^{w_n})=P_i+(x_{n-1})$,
where $P_i=(x_2^{w_2},x_2x_3^{w_3},\ldots,x_{n-3}x_{n-2}^{w_{n-2}})+\sum\limits_{j=1}^{p'_i}(x_{i_j+1}^{w_{i_j+1}})$,
where  if $i_{p_i-1}=n-2$, then  ${p'_i}=p_{i}-2$ otherwise, ${p'_i}=p_{i}-1$.

Let $P_i^{\mathcal {P}}$ be the polarization of the ideal $P_i$, then  $|\supp\,(P_i^{\mathcal {P}})|=\sum\limits_{j=2}^{n-2}w_j$ and $|\mathcal {G}(P_i^{\mathcal {P}})|=n-3$.  Similar arguments as above, we have
\begin{eqnarray*}
\mbox{reg}\,((K'_i:x_n^{w_n})(-w_n))\!\!&=&\!\!\!\mbox{reg}\,(P_i+(x_{n-1}))+w_n=\mbox{reg}\,(P_i)+w_n=\mbox{reg}\,(P_i^{\mathcal {P}})+w_n\\
\!\!&=&\!\!\!\sum\limits_{j=2}^{n-2}w_j-(n-3)+1+w_n\leq \sum\limits_{j=2}^{n}w_j-n+2  \hspace{2.0cm}(6)
\end{eqnarray*}
where the inequality holds because of $w_{n-1}\geq2$.

Using formulas  (5) and (6), Lemma \ref{lem2}  and Lemma \ref{lem6} (1)  on the short exact sequence
$$0\longrightarrow \frac{S}{(K'_i:x_n^{w_n})}(-w_n)\stackrel{ \cdot x_n^{w_n}} \longrightarrow \frac{S}{K'_i}\longrightarrow \frac{S}{T_0}\longrightarrow 0, $$
we have
$$\mbox{reg}\,(T_0)\leq\sum\limits_{j=2}^{n}w_j-n+1. $$

Suppose the formulas  (4) is  true for any $1\leq j\leq q_i-1$. Now assume $j=q_i$.
We first compute $(T_{q_i-1}:u_{q_i})$. Since   $K_i=((L_{i_1+1},\ldots, L_n):L_{i_1})+\sum\limits_{j=1}^{p_i}(L_{i_j+1}:L_{i_j})$ and
$T_{q_i-1}=K_i+(u_0,u_{1},\ldots,u_{q_i-1})$,  we obtain by simple calculation
\begin{eqnarray*}
& &((\sum\limits_{j=0}^{q_i-1}u_j): u_{q_i})=\sum\limits_{j=0}^{q_i-1}(u_j: u_{q_i})=\sum\limits_{j=0}^{q_i-1}((\prod\limits_{s=0}^{j}L_{n-2s}:\prod\limits_{s=0}^{j}L_{n+1-2s}): u_{q_i})\\
&=&\sum\limits_{j=0}^{q_i-1}(x_{n-2j-1})=(x_{n-2q_i+1},x_{n-2q_i+3},\ldots,x_{n-1}),
\end{eqnarray*}
\begin{eqnarray*}
& &((\sum\limits_{j=1}^{p_i}(L_{i_j+1}:L_{i_j})): u_{q_i})=\sum\limits_{j=1}^{p_i}((L_{i_j+1}:L_{i_j}):u_{q_i})=\sum\limits_{j=1}^{p_i}(x_{i_j+1}^{w_{i_j+1}}:u_{q_i})\\
&=&\left\{\begin{array}{ll}
(x_2)+\sum\limits_{j=0}^{q_i-1}(x_{n-2s}),&\ \text{if}\ n \ \text{is even and}\ q_i=\lfloor\frac{n}{2}\rfloor-1,\\
\sum\limits_{j=1}^{p_i}(x_{i_j+1}^{w_{i_j+1}})+\sum\limits_{j=0}^{q_i-1}(x_{n-2s}),&\ \text{otherwise},\\
\end{array}\right.
\end{eqnarray*}
\begin{eqnarray*}
& &(T_{q_i-1}:u_{q_i})=((K_i+(u_0,u_{1}\ldots,u_{q_i-1})):u_{q_i})\\
&=&((\sum\limits_{j=i_1+1}^{n}L_{j}):L_{i_1}):u_{q_i})+(\sum\limits_{j=1}^{p_i}(L_{i_j+1}:L_{i_j}):u_{q_i})+((\sum\limits_{j=0}^{q_i-1}u_j): u_{q_i})\\
&=&\sum\limits_{j=i_1+1}^{n}((L_{j}:L_{i_1}):u_{q_i})+\sum\limits_{j=1}^{p_i}((L_{i_j+1}:L_{i_j}):u_{q_i})+((\sum\limits_{j=0}^{q_i-1}u_j): u_{q_i})\\
&=&\left\{\begin{array}{ll}
\hspace{-0.2cm}(A+(x_2))+\sum\limits_{j=3}^{n}(x_j),&\text{if}\ n\ \text{is even,}\ q_i=\lfloor\frac{n}{2}\rfloor-1\\
\hspace{-0.2cm}(A+(x_{2}^{w_{2}-1},x_3))+\!\!\sum\limits_{j=4}^{n}(x_j),&\text{if}\ n\ \text{is odd,}\ q_i=\lfloor\frac{n}{2}\rfloor-1\\
\hspace{-0.2cm}(A+B+(x_{n-2q_i}))+\!\!\sum\limits_{j=1}^{p_i}(x_{i_j+1}^{w_{i_j+1}})+\!\!\sum\limits_{j=n-2q_i+1}^{n}(x_{j}),&\text{otherwise},\\
\end{array}\right.\\
&=&\left\{\begin{array}{ll}
\sum\limits_{j=2}^{n}(x_j),&\text{if}\ n \ \text{is even and}\ q_i=\lfloor\frac{n}{2}\rfloor-1\\
(x_{2}^{w_{2}-1})+\sum\limits_{j=3}^{n}(x_j),&\text{if}\ n \ \text{is odd and}\ q_i=\lfloor\frac{n}{2}\rfloor-1\\
B+\sum\limits_{j=1}^{p''_i}(x_{i_j+1}^{w_{i_j+1}})+\sum\limits_{j=n-2q_i}^{n}(x_{j}),&\text{otherwise},\\
\end{array}\right.
\end{eqnarray*}
where $A=\sum\limits_{j=2}^{q_i}(x_{n-2j+1}^{w_{n-2j+1}},x_{n-2j+1}x_{n-2j+2})+(x_{n-1})$, $B=(x_{2}^{w_{2}},x_{n-2q_i-2}x_{n-2q_i-1}^{w_{n-2q_i-1}-1})+\sum\limits_{j=2}^{n-2(q_i+1)}(x_{j}x_{j+1}^{w_{j+1}})$
and ${p''_i}=\mbox{max}\,\{{p''_i}: 1\leq p''_i\leq p_i\ \mbox{and } i_{p''_i}\leq  n-2q_i-2\}$.

Next we  compute $\mbox{reg}\,((T_{q_i-1}:u_{q_i}))$.  Let $d$ be the degree of monomial $u_{q_i}$. If $q_i=\lfloor\frac{n}{2}\rfloor-1$ and $n$ is even, then $d=\sum\limits_{j=1}^{\frac{n}{2}}w_{2j}-\frac{n}{2}$, otherwise, $d=\sum\limits_{j=0}^{q_i}w_{n-2j}-q_i$. We distinguish into the following three case:

(i) If $n=2m$  and  $q_i=m-1$,  then by Lemma \ref{lem3},
\begin{eqnarray*}
\mbox{reg}\,((T_{q_i-1}:u_{q_i})(-d))&=&\mbox{reg}\,((T_{q_i-1}:u_{q_i}))+d=\mbox{reg}\,(\sum\limits_{j=2}^{n}(x_j))+d\\
&=&1+(\sum\limits_{j=1}^{m}w_{2j}-m)=(\sum\limits_{j=2}^{n}w_{j}-n+1)+(m-\sum\limits_{j=2}^{m}w_{2j-1})\\
&\leq&\sum\limits_{j=2}^{n}w_j-n+1.
\end{eqnarray*}

(ii) If $n=2m+1$ and $q_i=m-1$, then by Lemma \ref{lem3} (1),
\begin{eqnarray*}
\mbox{reg}\,((T_{q_i-1}:u_{q_i})(-d))&=&\mbox{reg}\,((T_{q_i-1}:u_{q_i}))+d=\mbox{reg}\,((x_{2}^{w_{2}-1})+\sum\limits_{j=3}^{n}(x_j))+d\\
&=&(\sum\limits_{j=0}^{m-1}w_{n-2j}-(m-1))+(w_2-1)\\
&=&\!(\sum\limits_{j=2}^{n}w_{j}-n+1)+(m-\sum\limits_{j=1}^{m-1}w_{n-2j+1})\\
&\leq&\sum\limits_{j=2}^{n}w_j-n+1.
\end{eqnarray*}

(iii) In other cases, by Lemma \ref{lem3}, we have
\begin{eqnarray*}
& &\mbox{reg}\,((T_{q_i-1}:u_{q_i})(-d))=\mbox{reg}\,((T_{q_i-1}:u_{q_i}))+d\\
&=& \mbox{reg}\,(B+\sum\limits_{j=1}^{p''_i}(x_{i_j+1}^{w_{i_j+1}})+\sum\limits_{ j=n-2q_i}^{n}(x_{j}))+d
=\mbox{reg}\,(B+\sum\limits_{j=1}^{p''_i}(x_{i_j+1}^{w_{i_j+1}}))+d\\
&\leq&(\sum\limits_{j=2}^{n-2q_i-1}\!\!w_{j}-(n-2q_i-1)+1)+(\sum\limits_{j=0}^{q_i}w_{n-2j}-q_i)\\
&=&(\sum\limits_{j=2}^{n}w_{j}-n+1)+(q_i+1-\sum\limits_{j=1}^{q_i}w_{n-2j+1})\leq \sum\limits_{j=2}^{n}w_j-n+1,
\end{eqnarray*}
where the first inequality holds because of $\mbox{reg}\,(B+\sum\limits_{j=1}^{p''_i}(x_{i_j+1}^{w_{i_j+1}}))\leq
\sum\limits_{j=2}^{n-2q_i-1}\!\!w_{j}-(n-2q_i-1)+1$  by similar arguments as the calculation of $\mbox{reg}\,(T_0)$.

Using the above formulas of $\mbox{reg}\,((T_{q_i-1}:u_{q_i})(-d))$,  Lemma \ref{lem2}, Lemma \ref{lem6} (1) and the induction hypothesis on the short exact sequence
$$0\longrightarrow \frac{S}{(T_{q_i-1}:u_{q_i})}(-d)\stackrel{ \cdot u_{q_i}} \longrightarrow \frac{S}{T_{q_i-1}}\longrightarrow \frac{S}{T_{q_i}}\longrightarrow 0,$$
we have
$$\mbox{reg}\,((J_{i}:L_{i}^{(t)}))=\mbox{reg}\,(T_{q_i})\leq \sum\limits_{j=2}^{n}w_j-n+1. $$
The proof is complete.
\end{proof}

\medskip
The following Theorem is main result in this section.
¡®\begin{Theorem}\label{thm4}
Let  $C_n=(V(C_n),E(C_n),w)$ be  a vertex-weighted oriented cycle,  $I(C_n)=(L_1,\ldots,L_n)$   an edge ideal of $C_n$,  where $L_i=x_{i-1}x_i^{w_i}$ and $w_i\geq 2$ for $1\leq i\leq n$. Then
$$\mbox{reg}\,(I(C_n)^{t})=\sum\limits_{x\in V(C_n)}w(x)-|E(C_n)|+1+(t-1)(w+1)\ \ \  \mbox{for any}\ \ t\geq 1,$$
where $w=\mbox{max}\,\{w_i\mid 1\leq i\leq n\}$.
\end{Theorem}
\begin{proof}
Case $t=1$ follows from \cite[Theorem 4.1]{Z4}.
Now assume $t\geq2$.  Let $w=w_1$ without loss of generality and
$L^{(t)}=\{L_1^{(t)}, \ldots, L_r^{(t)}\}$  a totally ordered set of all elements of
  $\mathcal{G}(I(C_n)^{t})$ such that $L_1^{(t)}> \cdots >L_r^{(t)}$.
 For  $1\leq i\leq r$, we write $L_i^{(t)}$ as
$L_i^{(t)}=L_{i_1}^{a_{i_1}}\cdots L_{i_{k_i}}^{a_{i_{k_i}}}$ with $1\leq i_1<\cdots < i_{k_i}\leq n$,
$\sum\limits_{j=1}^{{k_i}}a_{i_j}=t$ and $a_{i_j}>0$ for $j=1,\ldots,k_i$. Let $d_i$ be the degree of monomial $L_{i}^{(t)}$, then  we get
$d_i\leq(w_{i_1}+1)+(t-1)(w+1)$ for $1\leq i\leq r-1$ by the definition of $w$. We prove this argument in the
following two steps.

Step $1$:  We first show $\mbox{reg}\,(I(C_n)^{t})\leq\sum\limits_{j=1}^{n}w_j-n+1+(t-1)(w+1)$. Let  $J_{i}=(L_{i+1}^{(t)},\ldots, L_{r}^{(t)})$
for $1\leq i\leq r-1$. Since $J_{r-1}=(L_{r}^{(t)})=(x_{n-1}^tx_n^{tw_n})$, we get
\begin{eqnarray*}
\mbox{reg}\,(J_{r-1})&=&t(w_n+1)=(\sum\limits_{j=1}^{n}w_j-n+1+(t-1)(w_n+1))+(n-\sum\limits_{j=1}^{n-1}w_j)\\
&\leq&\sum\limits_{j=1}^{n}w_j-n+1+(t-1)(w+1), \hspace{5.5cm}(1)
\end{eqnarray*}
where the inequality above holds because of $w_n\leq w$ and $w_j\geq2$ for  $1\leq j\leq n-1$.

By Proposition \ref{prop1} and Proposition \ref{prop2}, we have, for any $1\leq i\leq r-1$,
\begin{eqnarray*}
\mbox{reg}\,((J_i:L_{i}^{(t)}))&\leq & \left\{\begin{array}{ll}
\sum\limits_{j=2}^{n}w_j-n+1\ \ &\text{if}\ \ i_1=1,\\
 \sum\limits_{j=i_1+1}^{n}w_j-(n-i_1)+1\ \ &\text{if} \ \ i_1\geq 2,\\
\end{array}\right.\\
\end{eqnarray*}
\begin{eqnarray*}
&\leq&\left\{\begin{array}{ll}
\sum\limits_{j=1}^{n}w_j-n+1-w_1\ \ &\text{if}\ \ i_1=1,\\
 \sum\limits_{j=1}^{n}w_j-n+1-w_{i_1}\ \ &\text{if} \ \ i_1\geq 2,\\
\end{array}\right.
\end{eqnarray*}
where the last inequality holds because of $w_j\geq2$ for  $1\leq j\leq n$.
It follows that
\begin{eqnarray*}
\mbox{reg}\,((J_i:L_{i}^{(t)})(-d_i))\!\!\!&=&\!\!\!\mbox{reg}\,((J_i:L_{i}^{(t)}))+d_i\\
\!\!\!&\leq&\!\!\! (\sum\limits_{j=1}^{n}w_j-n+1-w_{i_1})+((w_{i_1}+1)+(t-1)(w+1))\\
\!\!\!&=&\!\!\!\sum\limits_{j=1}^{n}w_j-n+1+(t-1)(w+1)+1.\hspace{3.7cm}(2)
\end{eqnarray*}

Using the formulas (1) and (2), Lemma \ref{lem2} and Lemma \ref{lem6} (1)  on the following short exact sequences
\begin{gather*}
\hspace{1.0cm}\begin{matrix}
 0 & \longrightarrow & \frac{S}{(J_1:L_{1}^{(t)})}(-d_1) & \stackrel{ \cdot L_{1}^{(t)}} \longrightarrow & \frac{S}{J_1} &\longrightarrow & \frac{S}{I(C_n)^{t}} & \longrightarrow & 0  \\
 0 & \longrightarrow & \frac{S}{(J_2:L_{2}^{(t)})}(-d_2)  & \stackrel{ \cdot L_{2}^{(t)}} \longrightarrow  & \frac{S}{J_2} & \longrightarrow & \frac{S}{J_1} & \longrightarrow & 0 & \\
 \\
     &  &\vdots&  &\vdots&  &\vdots&  &\\
 0&  \longrightarrow & \frac{S}{(J_{r-1}:L_{r-1}^{(t)})}(-d_r) & \stackrel{ \cdot L_{r-1}^{(t)}} \longrightarrow & \frac{S}{J_{r-1}}& \longrightarrow & \frac{S}{J_{r-2}}& \longrightarrow & 0,
 \end{matrix}
\end{gather*}
we obtain
$$\mbox{reg}\,(I(C_n)^{t})\leq\sum\limits_{j=1}^{n}w_j-n+1+(t-1)(w+1).\eqno(3)$$

Step $2$:  We  show $\mbox{reg}\,(I(C_n)^{t})=\sum\limits_{j=1}^{n}w_j-n+1+(t-1)(w+1)$.

We write $I(C_n)^t$ as  $I(C_n)^t=J+K$ with $\mathcal{G}(I(C_n)^t)=\mathcal{G}(J)\bigsqcup \mathcal{G}(K)$ and  $K=(L_1^{(t)})$. Let
$J^{\mathcal{P}}$, $K^{\mathcal{P}}$ and $(I(C_n)^t)^{\mathcal{P}}$ be the polarization of $J$, $K$ and $(I(C_n)^t)$ respectively, then $K=(x_n^tx_1^{tw_1})$ and
\[(I(C_n)^t)^{\mathcal{P}}=J^{\mathcal{P}}+K^{\mathcal{P}}\ \ \ \  \text{and}\ \ \ \  J^{\mathcal{P}}\cap K^{\mathcal{P}}=K^{\mathcal{P}}L,
\]
where $L=(\prod\limits_{j=1}^{w_{2}}x_{2j},x_{21}\!\prod\limits_{j=1}^{w_{3}}\!x_{3j},\ldots, x_{n-1,\,1}\!\!\prod\limits_{j=t+1}^{t-1+w_{n}}\!\!x_{n,\,j})$.
Then $|\supp\,(L)|=\sum\limits_{j=2}^{n}w_j-1$ and $|\mathcal {G}(L)|=n-1$. We distinguish into the following two steps:

Step (i): We first compute $\mbox{reg}\,(J^{\mathcal{P}}\cap K^{\mathcal{P}})$.

Since $x_{2,w_2}$ (resp.  $x_{n,t-1+w_n}$)  is only a factor of the unique monomial $\prod\limits_{j=1}^{w_{2}}x_{2j}$ (resp. $x_{n-1,\,1}\!\!\prod\limits_{j=t+1}^{t-1+w_{n}}\!\!x_{n,\,j}$)  of  the set $\mathcal {G}(L)$ and $x_{j,w_j}$ is also only a factor of the unique monomial
$x_{j-1,1}\prod\limits_{k=1}^{w_{j}}x_{j,k}$ of  the set $\mathcal {G}(L)$ for any $3\leq j\leq n-1$ and  the variables that appear in  $K^{\mathcal{P}}$ and  $L$ are different, thus  by Lemma \ref{lem4} (2) and Lemma \ref{lem9}, we obtain
\begin{eqnarray*}
\hspace{2cm}\mbox{reg}\,( J^{\mathcal{P}}\cap K^{\mathcal{P}})\!\!\!&=&\!\!\!\mbox{reg}\,(K^{\mathcal{P}}L)
=\mbox{reg}\,(K^{\mathcal{P}})+\mbox{reg}\,(L)\\
\!\!\!&=&\!\!\!t(w+1)+(|\supp\,(L)|-|\mathcal {G}(L)|+1)\\
\!\!\!&=&\!\!\!t(w+1)+(\sum\limits_{j=2}^{n}w_j-1-(n-1)+1)\\
&=&\sum\limits_{j=1}^{n}w_j-n+1+(t-1)(w+1)+1. \hspace{2.5cm}(4)
\end{eqnarray*}

Step (ii): We  compute $\mbox{reg}\,(J^{\mathcal{P}})$.

Let $H=(V(H), \mathcal{E}(H))$ and $H'=(V(H'), \mathcal {E}(H'))$ are hypergraphs associated to $\mathcal{G}((I(C_n)^t)^{\mathcal{P}})$ and $\mathcal{G}(J^{\mathcal{P}})$ respectively, then  $H'$ is an induced subhypergraph of $H$.
In fact, $H'$ is a  subhypergraph of $H$ by the choice  of $\mathcal{G}((I(C_n)^t)^{\mathcal{P}})$ and $\mathcal{G}(J^{\mathcal{P}})$.
On the other hand, if $E\in \mathcal{E}(H)$ with  $E\subseteq V(H')$, then  monomial $\prod\limits_{x_{ij}\in E}x_{ij}$ associated to $E$ belong to $\mathcal{G}((I(C_n)^t)^{\mathcal{P}})$.
Since $\mathcal{G}((I(C_n)^t)^{\mathcal{P}})=\mathcal{G}(K^{\mathcal{P}})\bigsqcup \mathcal{G}(J^{\mathcal{P}})$,
if $\prod\limits_{x_{ij}\in E}x_{ij}\in  \mathcal{G}(K^{\mathcal{P}})$, then $x_{1,tw_1}\in E$ by definition of $\mathcal{G}(K^{\mathcal{P}})$,
contradicting with  $x_{1,tw_1}\notin V(H')$. Thus $\prod\limits_{x_{ij}\in E}x_{ij}\in \mathcal{G}(J^{\mathcal{P}})$.
Hence $H'$ is an induced subhypergraph of $H$. By  Lemma \ref{lem5} (2), Lemma \ref{lem8} and the formula (3), we get
$$\mbox{reg }\,(J^{\mathcal{P}})\leq\mbox{reg }\,((I(C_n)^t)^{\mathcal{P}})
=\mbox{reg }\,((I(C_n)^t))\leq \sum\limits_{j=1}^{n}w_j-n+1+(t-1)(w+1).\eqno(5)$$
Let $\alpha=\mbox{reg}\,( J^{\mathcal{P}}\cap K^{\mathcal{P}})-1$ and $\beta=\mbox{reg}\,(K^{\mathcal{P}})=t(w+1)$, then
$$\alpha-\beta=(\sum\limits_{j=1}^{n}w_j-n+1+(t-1)(w+1))-t(w+1)
= \sum\limits_{j=2}^{n}w_j-n \geq 0,\eqno(6)$$
where the inequality holds because of $w_j\geq2$ for  $2\leq j\leq n$.

Since the variable $x_{1,tw_1}$ in $\supp\,(K^{\mathcal {P}})$  can not divided generators of $J^{\mathcal {P}}$ and   $K^{\mathcal {P}}$ has  a linear resolution. By Lemma \ref{lem1},  it follows that
$(I(C_n)^t)^{\mathcal{P}}=J^{\mathcal{P}}+K^{\mathcal{P}}$ is Betti splitting.  By Corollary \ref{cor1}, formulas (4), (5) and  (6), we obtain
\begin{eqnarray*}
\mbox{reg}\,((I(C_n)^t))\!\!\!&=&\!\!\!\mbox{reg}\,((I(C_n)^t)^{\mathcal{P}})=\mbox{max}\, \{\mbox{reg}\,(J^{\mathcal{P}}),\mbox{reg}\,(K^{\mathcal{P}}), \mbox{reg}\,(J^{\mathcal{P}}\cap K^{\mathcal{P}})-1\}\\
\!\!\!&=&\!\!\!\mbox{max}\, \{\mbox{reg}\,(J^{\mathcal{P}}),t(w+1),(\sum\limits_{j=1}^{n}\!w_j-n+1+(t-1)(w+1)+1)-1\}\\
\!\!\!&=&\!\!\!\sum\limits_{j=1}^{n}w_j-n+1+(t-1)(w+1).
\end{eqnarray*}
This proof is completed.
\end{proof}

\medskip
 As a consequence of Theorem \ref{thm4}, we have
\begin{Corollary}\label{cor2}
Let $C_n=(V(C_n),E(C_n),w)$ be a vertex-weighted oriented cycle as in Theorem \ref{thm4}. Then
$$\mbox{reg}\,(I(C_n)^{t})=\mbox{reg}\,(I(C_n))+(t-1)(w+1)\hspace{1cm}  \mbox{for any}\ \ t\geq 1,$$
where $w=\mbox{max}\,\{w(x)\mid x \in V(C_n)\}$.
\end{Corollary}

\medskip
The following example shows  the assumption  that  $w(x)\geq 2$ for any  $x\in V(C_n)$  in Theorem \ref{thm4}  cannot be dropped.
\begin{Example}  \label{example6}
Let $I(C_5)=(x_5x_1,x_1x_2^3,x_2x_3^3,x_3x_4,x_4x_5^3)$ be an edge ideal of the  vertex-weighted  oriented cycle $C_5=(V,E,w)$, its weight function is  $w_2=w_3=w_5=3$,  $w_1=w_4=1$. Thus  $w=3$. By using CoCoA, we obtain $\mbox{reg}\,(I(C_5)^2)=10$. But we have $\mbox{reg}\,(I(C_5)^2)=\sum\limits_{i=1}^{5}w_i-|E(C_5)|+1+w+1=11$  by Theorem \ref{thm4}.
\end{Example}

\medskip
The following example shows that the
regularity  of powers of  edge ideals of vertex-weighted oriented  cycles as Theorem \ref{thm4} is related to direction selection.
\begin{Example}  \label{example7}
Let $I(C_5)=(x_1x_5^3,x_1x_2^3,x_2x_3^3,x_3x_4^3,x_4x_5^3)$ be an  edge ideal of the vertex-weighted  oriented cycle $C_5=(V,E,w)$ with $w_2=w_3=w_4=w_5=3$,  $w_1=1$. Thus $w=3$. By using CoCoA, we obtain $\mbox{reg}\,(I(C_5)^2)=14$. But we have $\mbox{reg}\,(I(C_5)^2)=\sum\limits_{i=1}^{5}w_i-|E(C_5)|+1+w+1=13$  by Theorem \ref{thm4}.
 \end{Example}

\medskip

\section{Regularity of  powers of edge ideals of vertex-weighted unicyclic graphs}

 In this section, we  consider a vertex-weighted  oriented
unicyclic graph $D=(V(D),E(D),w)$ satisfying its underlying graph $G$ is the union of a circle and some forests. We will provide the exact formulas for the regularity of  powers of its edge ideal. We also give some examples to show the  regularity of powers of  edge ideals of vertex-weighted oriented unicyclic graphs is related to direction selection and  the assumption  that
$w(x)\geq 2$  if $d(x)\neq 1$ cannot be dropped.

 \begin{Definition}\label{def2}
 Let $G_i=(V_i,E_i)$ be some simple graphs for $1\leq i\leq s$,  their union is a graph $G=(V,E)$,  denoted by $\bigcup\limits_{i=1}^{s}G_{i}$, satisfying  its vertex set is $V=\bigcup\limits_{i=1}^{s}V_{i}$ and its edge set is $E=\bigcup\limits_{i=1}^{s}E_{i}$.
\end{Definition}

\begin{Definition}\label{def3}
 Let $G=(V(G),E(G))$ be a unicyclic graph with $n$ vertices. We write $G$ as $G=G_0\cup(\bigcup\limits_{j=1}^{s}G_j)$, where $G_0$ is an $m$-cycle and
 $G_j$ is a tree for $1\leq j\leq s$. The digraph $D=(V(D),E(D),w)$ is called  an oriented  unicyclic graph, denoted by $D=C_m\cup(\bigcup\limits_{j=1}^{s}T_j)$, if its underlying graph is $G$,  and $C_m$ is  an oriented cycle with underlying graph $G_0$ and $T_j$ is an oriented tree with  underlying graph $G_j$, its  orientation  is as follows:  if $V(G_0)\cap V(G_j)=\{x_{i_j}\}$, then $x_{i_j}$ is the root of $T_j$, and  all edges in $T_j$ are oriented away from $x_{i_j}$ for $1\leq j\leq s$.
\end{Definition}

Throughout  this section, let $D=(V(D),E(D),w)$ be a vertex-weighted oriented unicyclic graph with vertex set
 $V(D)=\{x_1,\ldots,x_n\}$, where  $C_m$ is the unique oriented cycle in $D$, its vertex set
 $V(C_m)=\{x_1,\ldots,x_m\}$, its edge set $E(C_m)=\{x_1x_2^{w_2},\ldots,x_mx_1^{w_1}\}$.
 The   orientation of $D$ defined  as above and the
weight $w(x_i)\geq 2$ of $x_i$  if $d(x_i)\neq 1$ for $1\leq i\leq n$.

\medskip
 Let  $D=(V(D), E(D),w)$ be a vertex-weighted oriented graph. For $T\subset V(D)$, we define
the {\em induced vertex-weighted  subgraph} $H=(V(H), E(H),w)$ of $D$  to be a vertex-weighted oriented graph
with $V(H)=T$,  for any $u,v\in V(H)$, $uv\in E(H)$  if and only if $uv\in E(D)$ and its orientation in $H$  is the same as  in $D$. For any $u\in V(H)$ and $u$ is not a source in $H$, its weight in $H$ equals to the weight of $u$ in $D$, otherwise, its weight in $H$ equals to $1$.
For  $P\subset V(D)$, we  denote
$D\setminus P$ the induced subgraph of $D$ obtained by removing the vertices in $P$ and the
edges incident to these vertices. If $P=\{x\}$ consists of a  element, then we write $D\setminus x$ for $D\setminus \{x\}$.
If $x\in V(D)$, then we denote by $N_D^{+}(x)=\{y:(x,y)\in E(D)\}$, $N_D^{-}(x)=\{y:(y,x)\in E(D)\}$ and $N_D(x)=N_D^{+}(x)\cup N_D^{-}(x)$.

\medskip
We need the following two lemmas, see for instance \cite[Lemma 3.4, Lemma 3.5, Lemma 3.6 and Theorem 4.2]{Z5}.
\begin{Lemma}\label{lem10}
 Let  $t\geq 2$ be a positive integer and $D=(V(D),E(D),w)$  a vertex-weighted oriented graph, let $z$ be a leaf with $N_D^{-}(z)=\{y\}$. Then,
\begin{itemize}
\item[(1)] $(I(D)^{t},z^{w_z})=(I(D\setminus z)^{t},z^{w_z})$,
\item[(2)] $(I(D)^{t}:yz^{w_z})=I(D)^{t-1}$,
\item[(3)] $((I(D)^{t}:z^{w_z}),y)=((I(D\setminus y)^{t}:z^{w_z}),y)=(I(D\setminus y)^{t},y)$.
\end{itemize}
\end{Lemma}

\medskip
\begin{Lemma}\label{lem11}Let   $D=(V(D),E(D),w)$ be  a vertex-weighted rooted forest such that   $w(x)\geq 2$  if $d(x)\neq 1$. Let $w=\mbox{max}\,\{w(x)\mid x \in V(D)\}$, then
\[
\mbox{reg}\,(I(D)^{t})=\sum\limits_{x\in V(D)}w(x)-|E(D)|+1+(t-1)(w+1).
\]
\end{Lemma}

\medskip
We need the following propositions to prove the main results.
\begin{Proposition}\label{prop4}
Let  $t$ be a positive integer and $D=(V(D),E(D),w)$  a vertex-weighted oriented unicyclic graph,
where $D=C_m\cup T_1$ and $T_1$ is an oriented line graph, its  orientation is  as follows: $x_i$ is the root of $T_1$ if $V(T_1)\cap V(C_m)=\{x_{i}\}$
for some $1\leq i\leq m$, otherwise, $x_{m+1}$ is the root of $T_1$.
Then
$$\mbox{reg}\,(I(D)^t)\leq \sum\limits_{x\in V(D)}w(x)-|E(D)|+1+(t-1)(w+1)\ \ \mbox{for any}\ \ t\geq 1$$
where $w=\mbox{max}\,\{w(x)|x\in V(D)\}$.
\end{Proposition}
\begin{proof}  Let $V(D)=\{x_1,\ldots,x_m,x_{m+1},\ldots,x_n\}$, $V(C_m)=\{x_1,\ldots,x_m\}$ and $w_i=w(x_i)$ for $1\leq i\leq n$.

If  $V(T_1)\cap V(C_m)=\emptyset$, then  the result can be shown by similar arguments as case $V(T_1)\cap V(C_m)=\{x_{i}\}$ for some $1\leq i\leq m$, so we only prove that the conclusion holds under the condition that
 $V(T_1)\cap V(C_m)=\{x_{i}\}$ for some $1\leq i\leq m$. In this case, we set $i=m$ for convenience.
Thus $E(D)=\{x_1x_2,x_2x_3,\ldots,x_{m-1}x_m,x_mx_1,x_mx_{m+1},\\
x_{m+1}x_{m+2},\ldots,x_{n-1}x_{n}\}$ and  $x_n$ is the unique leaf of $D$. It follows that
\[
I(D)=(x_1x_2^{w_2},\ldots,x_{m-1}x_m^{w_m},x_mx_1^{w_1},x_mx_{m+1}^{w_{m+1}},x_{m+1}x_{m+2}^{w_{m+2}},\ldots,x_{n-1}x_{n}^{w_n}).
\]
We apply induction on  $t$ and $|E(T_1)|$. Case $t=1$ follows from
 \cite[Theorem 3.4]{Z6}. Now assume that $t\geq 2$.
 If $|E(T_1)|=1$, then $n=m+1$.  Consider the following two short exact sequences
 $$0\longrightarrow \frac{S}{(I(D)^t:x_{n}^{w_{n}})}(-w_{n})\stackrel{ \cdot x_{n}^{w_{n}}} \longrightarrow \frac{S}{I(D)^t}\longrightarrow \frac{S}{(I(D)^t,x_{n}^{w_{n}})}\longrightarrow 0 \eqno(1)$$
$$0\!\longrightarrow \!\frac{S}{(I(D)^t:x_{m}x_{n}^{w_{n}})}(-1)\stackrel{\cdot x_{m}} \longrightarrow \! \frac{S}{(I(D)^t:x_{n}^{w_{n}})}\!\longrightarrow\! \frac{S}{((I(D)^t:x_{n}^{w_{n}}),x_{m})} \!\longrightarrow \! 0.\eqno(2)$$
Notice that  $D\setminus x_m$ is a vertex-weighted rooted forest.  By Lemma \ref{lem10}, we have  $(I(D)^t,x_{n}^{w_{n}})=(I(C_{m})^t,x_{n}^{w_{n}})$, $((I(D)^t:x_{n}^{w_{n}}),x_{m})=(I(D\setminus x_m)^t,x_{m})$ and $(I(D)^t:x_{m}x_{n}^{w_{n}})=I(D)^{t-1}$.
 Thus by Lemma  \ref{lem3}, Lemma  \ref{lem11}, Theorem  \ref{thm4} and induction hypothesis on $t$, we obtain
\begin{eqnarray*}
\mbox{reg}\,((I(D)^t,x_{n}^{w_{n}}))\!\!&=&\!\!\mbox{reg}\,((I(C_{m})^t,x_{n}^{w_{n}}))
=\mbox{reg}\,(I(C_{m})^t)+\mbox{reg}\,((x_{n}^{w_{n}}))-1\\
&=&(\sum\limits_{i=1}^{m}w_{i}-m+1+(t-1)(w'+1))+w_{n}-1\\
&=&\sum\limits_{x\in V(D)}w(x)-|E(D)|+1+(t-1)(w'+1)\\
&\leq&\sum\limits_{x\in V(D)}w(x)-|E(D)|+1+(t-1)(w+1),\hspace{2.2cm} (3)
\end{eqnarray*}
where the fourth equality holds because of $n=m+1$  and the last inequality holds because of $w'\leq w$, where $w'=\mbox{max}\,\{w_{i}\mid 1\leq i\leq m\}$,
\begin{eqnarray*}
& & \mbox{reg}\,((I(D)^t:x_{m}x_{n}^{w_{n}})(-w_{n}\!-\!1))=\mbox{reg}\,(I(D)^{t-1})+w_{n}+1\\
&=&\!\!\!\!(\sum\limits_{x\in V(D)}w(x)-|E(D)|+1+(t-2)(w+1))+w_{n}+1\\
&\leq&\!\!\!\!\sum\limits_{x\in V(D)}w(x)-|E(D)|+1+(t-1)(w+1),\hspace{5.2cm} (4)
\end{eqnarray*}
where the last inequality holds because of $w_{n}\leq w$,
and
\begin{eqnarray*}
& &\mbox{reg}\,(((I(D)^t:x_{n}^{w_{n}}),x_{m})(-w_{n}))=\mbox{reg}\,((I(D\setminus x_m)^t,x_{m}))+w_{n}\\
&=&\!\!\!\mbox{reg}\,(I(D\setminus x_m)^t)+w_{n}\\
&=&\!\!\!\!\sum\limits_{x\in V(D\setminus x_m)}w(x)-|E(D\setminus x_m)|+1+(t-1)(w''+1)+w_{n}\\
&=&\!\!\!\!\sum\limits_{x\in V(D)}w(x)-|E(D)|+1+\!(t-1)(w''+1)+4-(w_{1}+w_{m})\\
&\leq&\!\!\!\!\sum\limits_{x\in V(D)}w(x)-|E(D)|+1+(t-1)(w+1) \hspace{5.0cm} (5)
\end{eqnarray*}
where the forth equality holds because   we have weighted one in vertex $x_1$ in the expression $\sum\limits_{x\in V(D\setminus x_m)}w(x)$ and
$|E(D)|=|E(D\setminus x_m)|+3$,    and the last inequality holds because of $w_{1},w_{m}\geq 2$ and $w''\leq w$, here $w''=\mbox{max}\,\{w_i\mid 2\leq i \leq m-1\}$.
Using Lemma \ref{lem2} and  Lemma \ref{lem6} (2) on the short exact sequences (1), (2) and formulas (3)$\sim$ (5), we have
\[
\mbox{reg}\,(I(D)^t)\leq\sum\limits_{x\in V(D)}w(x)-|E(D)|+1+(t-1)(w+1).
\]

Assume $|E(T)|\geq 2$, consider the short exact sequences
 $$0\longrightarrow \frac{S}{(I(D)^t:x_{n}^{w_n})}(-w_{n})\stackrel{ \cdot x_{n}^{w_n}} \longrightarrow \frac{S}{I(D)^t}\longrightarrow \frac{S}{(I(D)^t,x_{n}^{w_n})}\longrightarrow 0 \eqno(6)$$
and
$$0\!\!\longrightarrow \!\frac{S}{((I(D)^t\!:\!x_{n-1}x_{n}^{w_n}))}(-1)\!\stackrel{\cdot x_{n-1}} \longrightarrow  \!\frac{S}{(I(D)^t\!:\!x_{n}^{w_n})}\!\!\longrightarrow \frac{S}{((I(D)^t\!:\!x_{n}^{w_n}),x_{n-1})} \!\!\longrightarrow\!  0.\eqno(7)$$
Notice that  $(I(D)^t,x_{n}^{w_n})=(I(D\setminus x_n)^t,x_{n}^{w_n})$, $((I(D)^t:x_{n}^{w_n}),x_{n-1})=(I(D\setminus x_{n-1})^t,x_{n-1})$ and $(I(D)^t:x_{n-1}x_{n}^{w_n})=I(D)^{t-1}$ by  Lemma \ref{lem10},
both $D\setminus x_n$ and $D\setminus x_{n-1}$ are vertex-weighted oriented unicyclic graphs. Thus,  by Lemma  \ref{lem3}, Theorem  \ref{thm4} and induction hypotheses on $t$ and $|E(T)|$, we obtain
\begin{eqnarray*}
\mbox{reg}\,((I(D)^t,x_{n}^{w_n}))&=&\mbox{reg}\,((I(D\setminus x_n)^t,x_{n}^{w_n}))=\mbox{reg}\,(I(D\setminus x_{n})^t)+\mbox{reg}\,((x_{n}^{w_n}))-1\\
\!\!\!&\leq&\!\!\!(\sum\limits_{x\in V(D\setminus x_n)}\!\!\!w(x)-|E(D\setminus x_n)|+1+(t-1)(w_a+1))+w_{n}-1\\
\!\!\!&=&\!\!\!\sum\limits_{x\in V(D)}\!\!w(x)-|E(D)|+1+(t-1)(w_a+1)\\
\!\!\!&\leq&\!\!\!\sum\limits_{x\in V(D)}\!\!w(x)-|E(D)|+1+(t-1)(w+1),\hspace{2.8cm} (8)
\end{eqnarray*}
where the last inequality holds because of $w_a\leq w$, where $w_a=\mbox{max}\,\{w(x)\mid x\in V(D\setminus x_n)\}$,
\begin{eqnarray*}
\mbox{reg}\,((I(D)^t:x_{n-1}x_{n}^{w_n})(-w_n\!\!-\!\!1))&=&\mbox{reg}\,(I(D)^{t-1})+w_n+1\\
\!\!\!&\leq&\!\!\!\!\!\!\sum\limits_{x\in V(D)}\!\!\!w(x)\!-\!|E(D)|\!+\!1+\!(t-2)(w+1)\!+\!w_n\!\!+\!1\\
\!\!\!&=&\!\!\!\!\!\!\sum\limits_{x\in V(D)}\!\!\!w(x)\!-\!|E(D)|\!+\!1+\!(t-1)(w+1)\!+\!w_n\!\!-\!w\\
\!\!\!&\leq&\!\!\!\!\!\!\sum\limits_{x\in V(D)}\!\!\!w(x)-|E(D)|+1+(t-1)(w+1),\hspace{0.5cm} (9)
\end{eqnarray*}
where the last inequality holds because of $w_n\leq w$,
and
\begin{eqnarray*}
& &\mbox{reg}\,(((I(D)^t\!\!:\!x_{n}^{w_n}\!),x_{n-1})(-w_n))=\mbox{reg}\,((I(D\setminus x_{n-1})^t,x_{n-1}))+w_n\\
\!\!\!&=&\!\!\!\mbox{reg}\,(I(D\setminus x_{n-1})^t)+w_n\leq\sum\limits_{x\in V(D\setminus x_{n-1})}\!\!\!\!\!\!\!\!\!w(x)\!\!-\!|E(D\setminus \!x_{n-1})|\!\!+\!1\!\!+\!(t\!-\!1)(w_b\!+\!1)\!\!+w_n\\
\!\!\!&=&\!\!\!\!\!\!\sum\limits_{x\in V(D)}\!\!\!w(x)\!-\!|E(D)|\!+\!1\!+\!(t\!-\!1)(w_b\!+\!1)\!+\!2\!-\!w_{n-1}\\
\!\!\!&\leq&\!\!\!\!\!\!\sum\limits_{x\in V(D)}\!\!\!w(x)-|E(D)|+1+(t-1)(w+1) \hspace{5.5cm} (10)
\end{eqnarray*}
where the last inequality holds because of $w_{n-1}\geq 2$, $w_b\leq w$, here $w_b=\mbox{max}\,\{w(x),x\in V(D\setminus x_{n-1})\}$.
Using  Lemma \ref{lem2} and  Lemma \ref{lem6} (2) on the short exact sequences (6) (7) and inequalities (8) $\sim$ (10), we have
\[
\mbox{reg}\,(I(D)^t)\leq\sum\limits_{x\in V(D)}w(x)-|E(D)|+1+(w+1).
\]
\end{proof}

\medskip
Now we are ready to present the main result of this section.
\begin{Theorem}\label{thm5}
Let $D=(V(D),E(D),w)$ be a vertex-weighted oriented unicyclic graph as  Proposition \ref{prop4}. Then
$$\mbox{reg}\,(I(D)^t)=\sum\limits_{x\in V(D)}w(x)-|E(D)|+1+(t-1)(w+1)\ \ \mbox{for any}\ \ t\geq 1$$
where $w=\mbox{max}\,\{w(x)|x\in V(D)\}$.
\end{Theorem}
\begin{proof}
Case $t=1$ follows from \cite[Theorem 3.5]{Z6}. Now we assume $t\geq 2$.
If  $V(T_1)\cap V(C_m)=\emptyset$, or $V(T_1)\cap V(C_m)=\{x_{i}\}$ for some $1\leq i\leq m$ and $|E(T_1)|\leq 3$, then  the conclusion can be shown by similar arguments as case $V(T_1)\cap V(C_m)=\{x_{i}\}$ for some $1\leq i\leq m$ and $|E(T_1)|\geq 4$, so we only prove the conclusion holds under the condition that
 $V(T_1)\cap V(C_m)=\{x_{i}\}$ for some $1\leq i\leq m$ and  $|E(T_1)|\geq 4$.
  In this case, we set $i=m$ for convenience.
Thus $E(D)=\{x_1x_2,x_2x_3,\ldots,x_{m-1}x_m,x_mx_1,x_mx_{m+1},x_{m+1}x_{m+2},\ldots,x_{n-1}x_{n}\}$. It follows that
\[
I(D)=(x_1x_2^{w_2},\ldots,x_{m-1}x_m^{w_m},x_mx_1^{w_1},x_mx_{m+1}^{w_{m+1}},x_{m+1}x_{m+2}^{w_{m+2}},\ldots,x_{n-1}x_{n}^{w_n}).
\]
Let $L$ be an ideal satisfying
\[
\mathcal{G}(L)=\mathcal{G}(I(D)^t)\setminus \mathcal{G}(M),
\]
where $M=((x_1x_2^{w_2})^t,\ldots,(x_{m-1}x_m^{w_m})^t,(x_mx_1^{w_1})^t,(x_mx_{m+1}^{w_{m+1}})^t
,\ldots,(x_{n-1}x_n^{w_n})^t)$.  Let $J_0$ be the polarization of $I(D)^t$, then
\[
J_0=M^{\mathcal{P}}+L^{\mathcal{P}}
\]
with $\mathcal{G}(J_0)=\mathcal{G}(M^{\mathcal{P}})\cup\mathcal{G}(L^{\mathcal{P}})$ and $\mathcal{G}(M^{\mathcal{P}})\cap\mathcal{G}(L^{\mathcal{P}})=\emptyset$.

For $1\leq i\leq n-2$, we set $K_{i}=((\prod\limits_{j=1}^{t}x_{i,\,j})(\prod\limits_{j=1}^{tw_{i+1}}x_{i+1,\,j}))$,
\begin{eqnarray*}
J_i&=&(\Widehat{(\prod\limits_{j=1}^{t}x_{1j})(\prod\limits_{j=1}^{tw_2}x_{2j})},\ldots,
\Widehat{(\prod\limits_{j=1}^{t}\!x_{i,\,j})(\prod\limits_{j=1}^{tw_{i+1}}x_{i+1,\,j})},
(\prod\limits_{j=1}^{t}\!x_{i+1,\,j})(\prod\limits_{j=1}^{tw_{i+2}}x_{i+2,\,j}),\ldots,\\
& &(\prod\limits_{j=1}^{t}\!\!x_{n-1,\,j})(\prod\limits_{j=1}^{tw_{n}}\!x_{n,\,j}),
(\prod\limits_{j=1}^{t}x_{m,\,j})(\prod\limits_{j=1}^{tw_1}x_{1j}))+L^{\mathcal{P}},
\end{eqnarray*}
 where $\Widehat{(\prod\limits_{j=1}^{t}\!x_{i,\,j})(\prod\limits_{j=1}^{tw_{i+1}}x_{i+1,\,j})}$ denotes the element
 $(\prod\limits_{j=1}^{t}\!x_{i,\,j})(\prod\limits_{j=1}^{tw_{i+1}}x_{i+1,\,j})$ being
 omitted from $J_i$.

Remind: when $i=m$, we set  $K_{m}=((\prod\limits_{j=1}^{t}x_{m,\,j})(\prod\limits_{j=1}^{tw_{m+1}}x_{m+1,\,j}))$,
\[
J_m=((\prod\limits_{j=1}^{t}x_{m+1,\,j})(\prod\limits_{j=1}^{tw_{m+2}}x_{m+2,\,j}),\ldots,
(\prod\limits_{j=1}^{t}x_{n-1,\,j})(\prod\limits_{j=1}^{tw_{n}}x_{n,\,j}),
(\prod\limits_{j=1}^{t}x_{m,\,j})(\prod\limits_{j=1}^{tw_1}x_{1j}))+L^{\mathcal{P}}.
\]
Let $K_{n-1}=((\prod\limits_{j=1}^{t}x_{n-1,\,j})(\prod\limits_{j=1}^{tw_{n}}x_{n,\,j}))$, $J_{n-1}=((\prod\limits_{j=1}^{t}x_{m,\,j})(\prod\limits_{j=1}^{tw_1}x_{1j}))+L^{\mathcal{P}}$,\\
$K_{n}=((\prod\limits_{j=1}^{t}x_{m,\,j})(\prod\limits_{j=1}^{tw_1}x_{1j}))$, $J_{n}=L^{\mathcal{P}}$.
Then for $1\leq i\leq n$, we have
$$J_{i-1}=J_{i}+K_{i}\ \  \text{and}\ \  J_{i}\cap K_{i}=K_{i}L_{i},$$
 \begin{eqnarray*}
L_1&=&(\prod\limits_{j=1}^{w_{3}}x_{3j},x_{31}\prod\limits_{j=1}^{w_{4}}x_{4j},\ldots,
x_{n-1,1}\prod\limits_{j=1}^{w_{n}}x_{n,j},x_{m,1}\prod\limits_{j=t+1}^{t-1+w_{1}}x_{1j}),\\
L_2&=&(x_{11}\prod\limits_{j=t+1}^{t-1+w_{2}}x_{2j},\prod\limits_{j=1}^{w_{4}}x_{4j},x_{41}\prod\limits_{j=1}^{w_{5}}x_{5j},\ldots,
x_{n-1,1}\prod\limits_{j=1}^{w_{n}}x_{n,j},x_{m,1}\prod\limits_{j=1}^{w_{1}}x_{1j}),\\
L_{i}&=&(x_{11}\prod\limits_{j=1}^{w_{2}}x_{2j},\ldots,x_{i-1,1}\!\!\!\!\prod\limits_{j=t+1}^{t-1+w_{i}}\!\!\!\!x_{i,\,j},
\prod\limits_{j=1}^{w_{i+2}}x_{i+2,\,j},x_{i+2,1}\prod\limits_{j=1}^{w_{i+3}}x_{i+3,\,j},\ldots,
x_{n-1,1}\prod\limits_{j=1}^{w_{n}}x_{n,j},\\
& &x_{m,1}\prod\limits_{j=1}^{w_{1}}x_{1j}) \ \ \text{for all}\ \ 3\leq i\leq m-2 \ \ \text{or}\ \ m+1\leq i\leq n-3,\\
L_{m-1}&=&(x_{11}\prod\limits_{j=1}^{w_{2}}x_{2j},\ldots,x_{m-3,1}\!\!\!\!\prod\limits_{j=1}^{w_{m-2}}\!\!\!x_{m-2,\,j},
x_{m-2,1}\!\!\!\!\prod\limits_{j=t+1}^{t-1+w_{m-1}}\!\!\!x_{m-1,\,j},
\prod\limits_{j=1}^{w_{m+1}}x_{m+1,\,j},\\
 & &x_{m+1,\,1}\prod\limits_{j=1}^{w_{m+2}}x_{m+2,\,j},\ldots,
x_{n-1,1}\prod\limits_{j=1}^{w_{n}}x_{n,j},\prod\limits_{j=1}^{w_{1}}x_{1j}),\\
\end{eqnarray*}
 \begin{eqnarray*}
 L_{m}&=&(x_{11}\prod\limits_{j=1}^{w_{2}}x_{2j},\ldots,x_{m-2,1}\!\!\!\!\prod\limits_{j=1}^{w_{m-1}}\!\!\!x_{m-1,\,j},
x_{m-1,1}\!\!\!\!\prod\limits_{j=t+1}^{t-1+w_{m}}\!\!\!x_{m,\,j},
\prod\limits_{j=1}^{w_{m+2}}x_{m+2,\,j},\\
& &x_{m+2,\,1}\prod\limits_{j=1}^{w_{m+3}}x_{m+3,\,j},\ldots,
x_{n-1,1}\prod\limits_{j=1}^{w_{n}}x_{n,j},\prod\limits_{j=1}^{w_{1}}x_{1j}),\\
L_{n-2}&=&(x_{11}\prod\limits_{j=1}^{w_{2}}x_{2j},\ldots,x_{n-4,1}\prod\limits_{j=1}^{w_{n-3}}x_{n-3,\,j},
x_{n-3,\,1}\!\!\!\!\prod\limits_{j=t+1}^{t-1+w_{n-2}}\!\!\!\!x_{n-2,\,j},\prod\limits_{j=1}^{w_{n}}x_{n,j},x_{m,1}\prod\limits_{j=1}^{w_{1}}x_{1j}),\\
L_{n-1}&=&(x_{11}\prod\limits_{j=1}^{w_{2}}x_{2j},\ldots,x_{n-3,1}\prod\limits_{j=1}^{w_{n-2}}x_{n-2,\,j},
x_{n-2,\,1}\prod\limits_{j=t+1}^{t-1+w_{n-1}}x_{n-1,\,j},x_{m,1}\prod\limits_{j=1}^{w_{1}}x_{1j}),
\end{eqnarray*}
\begin{eqnarray*}
L_{n}&=&(\prod\limits_{j=1}^{w_{2}}x_{2j},x_{21}\prod\limits_{j=1}^{w_{3}}x_{3j},\ldots,x_{m-2,1}\!\!\!\!\prod\limits_{j=1}^{w_{m-1}}\!\!\!x_{m-1,\,j},
x_{m-1,1}\!\!\!\!\prod\limits_{j=t+1}^{t-1+w_{m}}\!\!\!x_{m,\,j},
\prod\limits_{j=1}^{w_{m+1}}x_{m+1,\,j},\\
& &x_{m+1,\,1}\prod\limits_{j=1}^{w_{m+2}}x_{m+2,\,j},\ldots,
x_{n-1,1}\prod\limits_{j=1}^{w_{n}}x_{n,j}).
\end{eqnarray*}

Thus for $1\leq i \leq n$, $|\supp\,(L_{i})|=\sum\limits_{i=1}^{n}w(x)-w_{i+1}-1$ and $|\mathcal {G}(L_{i})|=n-1$.
By similar arguments as Proposition \ref{prop1}, we obtain $$\mbox{reg}\,(L_{i})=(\sum\limits_{i=1}^{n}w(x)-w_{i+1}-1)-(n-1)+1=\sum\limits_{i=1}^{n}w(x)-n+1-w_{i+1},$$
where $w_{n+1}=w_{1}$.  Notice that the  variables appear in $K_{i}$ and $L_{i}$ are different,  by Lemma   \ref{lem4} (2), we have
\begin{eqnarray*}
\mbox{reg }\,(J_{i}\cap K_{i})&=&\mbox{reg }\,(K_{i}L_{i})
=\mbox{reg }\,(K_{i})+\mbox{reg}\,(L_{i})\\
&=&t(w_{i+1}+1)+(\sum\limits_{i=1}^{n}w(x)-n+1-w_{i+1})\\
&=&\sum\limits_{i=1}^{n}w(x)-n+2+(t-1)(w_{i+1}+1).  \ \ \hspace{4.0cm} (1)
\end{eqnarray*}

Let $H=(V(H), \mathcal{E}(H))$ and $H'=(V(H'), \mathcal {E}(H'))$ are hypergraphs associated to $\mathcal{G}(J)$ and $\mathcal{G}(L^{\mathcal{P}})$ respectively,
then  $H'$ is an induced subhypergraph of $H$  by similar arguments as Theorem \ref{thm4}.  Thus by  Lemma \ref{lem5} (2), Lemma \ref{lem8} and Proposition \ref{prop4}, we get
\begin{eqnarray*}
\mbox{reg }\,(J_n)&=&\mbox{reg }\,(L^{\mathcal{P}})\leq \mbox{reg }\,((I(D)^t)^{\mathcal{P}})=\mbox{reg }\,(I(D)^t)\\
&\leq& \sum\limits_{x\in V(D)}w(x)-|E(D)|+1+(t-1)(w+1). \hspace{4.0cm} (2)
\end{eqnarray*}

For any $1\leq i \leq n$, the variable $x_{{i+1},tw_{i+1}}$ in $K_{i}$ is not a factor of any minimal generator of $J_{i}$ and $K_{i}$ has a linear resolution.
We have
$J_{i}=J_{i+1}+K_{i+1}$ is  Betti splitting by Lemma \ref{lem1}. Hence by Corollary \ref{cor1}, we obtain
\[
\mbox{reg}\,(J_{i-1})=\mbox{max}\{\mbox{reg}\,(K_{i}),\mbox{reg}\,(J_{i}), \mbox{reg}\,(K_{i}\cap J_{i})-1\}.\eqno(3)
\]
Let $\alpha=\mbox{reg}\,(K_{i})$, $\beta=\sum\limits_{x\in V(D)}w(x)-|E(D)|+1+(t-1)(w+1)$, then
\begin{eqnarray*}
\hspace{2.0cm} \mbox{reg}\,(K_{i})&=&t(w_{i+1}+1),\\
\beta-\alpha&=&(\sum\limits_{x\in V(D)}w(x)-|E(D)|+1+(t-1)(w+1))-t(w_{i+1}+1)\\
&\geq&\sum\limits_{\left.\begin{subarray}{l}x\in V(D)\\
x\neq x_{i+1}
\end{subarray}\right.}
w(x)-|E(D)|\geq 0.\hspace{5.2cm} (4)
\end{eqnarray*}
By Lemma \ref{lem5} (2), repeated use of the above the equality (3) and comparing formulas (1), (2), (4), we obtain
\begin{eqnarray*}
\mbox{reg}\,(I(D)^t)&=&\mbox{reg}\,(J_{0})=\mbox{max}\{\mbox{reg }\,(J_n),\mbox{reg}\,(K_{i}),\mbox{reg}\,(K_{i}\cap J_{i})-1,
 \  \text{for } \ 1\leq i \leq n \}\\
&=&\mbox{max}\{\mbox{reg }\,(J_n),t(w_{i+1}+1),\\
& &
\sum\limits_{x\in V(D)}w(x)-|E(D)|+2+(t-1)(w_i+1)-1, \  \text{for } \ 1\leq i \leq n\}\\
&=&\sum\limits_{x\in V(D)}w(x)-|E(D)|+1+(t-1)(w+1).
\end{eqnarray*}
The result follows.
\end{proof}

\begin{Theorem}\label{thm6}
Let $D=(V(D),E(D),w)$ be a vertex-weighted oriented unicyclic graph,
where $D=C_m\cup T$ and  $T$ is an oriented forest.  Let $w(x)\geq 2$ for any $d(x)\neq 1$.
Then
$$\mbox{reg}\,(I(D)^{t})=\sum\limits_{x\in V(D)}w(x)-|E(D)|+1+(t-1)(w+1)\ \ \  \mbox{for any}\ \ t\geq 1,$$
where $w=\mbox{max}\,\{w(x)\mid x \in V(D)\}$.
\end{Theorem}
\begin{proof}
 We apply  induction on  $t$ and $|E(T)|$. The case $t=1$ follows from \cite[Theorem 3.5]{Z6}. Now assume that $t\geq 2$.
By Theorem \ref{thm5}, we  just need to prove the results hold  under the condition that
  there are at least two leaves in $D$.
 Let  $x,z$ be leaves of $D$ with $w_z\leq w_x$ and $N_{D}^{-}(z)=\{y\}$.
 We distinguish into two cases:

(1) If there exists a connected component $T_1$ of $T$ such that $E(T_1)=\{yz\}$ and $y\notin V(C_m)$.
Then
$$I(D)^{t}=I(D\setminus z)^{t}+(yz^{w_z})I(D)^{t-1}.$$
Thus there exists a surjection $\phi :\ I(D\setminus z)^{t}\oplus  I(D)^{t-1}(-w_{z}-1)\overset{\cdot (1,\,yz^{w_z})}\longrightarrow  I(D)^{t}$
and the kernel of $\phi$ is  $(yz^{w_z})I(D\setminus z)^{t}$ since $yz^{w_z}$ is a non-zero divisor of $S/I(D\setminus z)$.
Therefore, we have the  following short exact sequence
$$0\longrightarrow I(D\setminus z)^{t}(-w_{z}-1)\longrightarrow I(D\setminus z)^{t}\oplus  I(D)^{t-1}(-w_{z}-1)\overset{\cdot (1,\,yz^{w_z})}\longrightarrow  I(D)^{t} \longrightarrow 0.$$
By induction hypotheses on $t$ and $|E(T)|$, we obtain
\begin{eqnarray*}\mbox{reg}\,(I(D)^{t-1}(-w_{z}-1))&=&\mbox{reg}\,(I(D)^{t-1})+w_{z}+1\\
&=&(\sum\limits_{x\in V(D)}w(x)-|E(D)|+1+(t-2)(w+1))+w_{z}+1\\
&=&\sum\limits_{x\in V(D)}w(x)-|E(D)|+1+(t-1)(w+1)+w_{z}-w,\\
\end{eqnarray*}
\begin{eqnarray*}
& &\mbox{reg}\,(I(D\setminus z)^{t}(-w_{z}-1))=\mbox{reg}\,(I(D\setminus z)^{t})+w_{z}+1\\
&=&(\sum\limits_{x\in V(D\setminus z)}w(x)-|E(D\setminus z)|+1+(t-1)(w+1))+w_{z}+1\\
&=&\sum\limits_{x\in V(D)}w(x)-|E(D)|+1+(t-1)(w+1)+1.
\end{eqnarray*}
Since  $w_z\leq w_x$, we get
\[
\mbox{reg}\,(I(D\setminus z)^{t}(-w_{z}-1))>\mbox{max}\,\{\mbox{reg}\,(I(D)^{t-1}(-w_{z}-1)),\mbox{reg}\,(I(D\setminus z)^{t})\}.
\]
Thus the result follows from Lemma \ref{lem6} (5).

(2)  If there is no connected component containing $z$ in  $T$ such as (1), then $d(y)\geq 2$. Consider the following short exact sequences
 $$0\longrightarrow \frac{S}{(I(D)^{t}:z^{w_z})}(-w_z)\stackrel{ \cdot z^{w_z}} \longrightarrow \frac{S}{I(D)^{t}}\longrightarrow \frac{S}{(I(D)^{t},z^{w_z})}\longrightarrow 0    \eqno(1)$$
$$0\longrightarrow \frac{S}{((I(D)^{t}:z^{w_z}):y)}(-1)\stackrel{\cdot y} \longrightarrow \frac{S}{(I(D)^{t}:z^{w_z})}\longrightarrow \frac{S}{((I(D)^{t}:z^{w_z}),y)} \longrightarrow 0 \eqno(2)$$
Note that $D\setminus z$ is  a vertex-weighted oriented unicyclic graph,  $D\setminus y$ is  an oriented unicyclic graph or a  rooted forest, and $w_z\leq w_x$,
thus, Lemma \ref{lem3}, Lemma \ref{lem11} or Theorem \ref{thm4} and induction hypotheses on  $t$ and $|E(T)|$, we obtain
\begin{eqnarray*}
& &\mbox{reg}\,((I(D)^t,z^{w_z}))=\mbox{reg}\,((I(D\setminus z)^{t},z^{w_z}))
=\mbox{reg}\,((I(D\setminus z)^{t}))+\mbox{reg}\,((z^{w_z}))-1\\
&=&[\sum\limits_{x\in V(D\setminus z)}\!\!w(x)-|E(D\setminus z)|+1+(t-1)(w+1)]+w_z-1\\
&=&\sum\limits_{x\in V(D)}w(x)-|E(D)|+1+(t-1)(w+1), \hspace{5.4cm}\  \ (3)
\end{eqnarray*}
\begin{eqnarray*}
& &\mbox{reg}\,((I(D)^{t}:yz^{w_z})(-w_{z}-1))=\mbox{reg}\,((I(D)^{t-1}))+w_{z}+1\\
&=&(\sum\limits_{x\in V(D)}w(x)-|E(D)|+1+(t-2)(w+1))+w_{z}+1\\
&=&\sum\limits_{x\in V(D)}w(x)-|E(D)|+1+(t-1)(w+1)+w_{z}-w\\
&\leq& \sum\limits_{x\in V(D)}w(x)-|E(D)|+1+(t-1)(w+1). \hspace{5.4cm}\ (4)
\end{eqnarray*}
\begin{eqnarray*}
& &\mbox{reg}\,(((I(D)^{t}:z^{w_z}),y)(-w_z))=\mbox{reg}\,(((I(D\setminus y)^t,y))+w_z\\
&=&\mbox{reg}\,((I(D\setminus y)^t))+w_z=(\!\!\!\!\sum\limits_{x\in V(D\setminus y)}\!\!\!\!w(x)-|E(D\setminus y)|+1+(t-1)(w''+1))+w_z,
\end{eqnarray*}
 where $w''=\mbox{max}\,\{w(x)\mid x \in V(D\setminus y)\}$.

Notice that   $N_D^{-}(y)=\emptyset$ or  $N_D^{-}(y)=\{y_1\}$. For case $N_D^{-}(y)=\emptyset$, it  can be shown by similar arguments as case $N_D^{-}(y)=\{y_1\}$. So we only prove  the conclusion holds under the condition that $N_D^{-}(y)=\{y_1\}$.
In this case, $w_y\geq2$ and we set $|E(D)|=|E(D\setminus y)|+\ell$, then $|N_{D}^{+}(y)\setminus \{z\}|=\ell-2$.
Let $\alpha=\mbox{reg}\,((I(D)^t,z^{w_z}))$, $\beta=\mbox{reg}\,(((I(D)^{t}:z^{w_z}),y)(-w_z)$, then
\begin{eqnarray*}
\alpha-\beta&=&\sum\limits_{x\in V(D)}w(x)-|E(D)|+1+(t-1)(w+1)\\
&-&[\sum\limits_{x\in V(D\setminus y)}w(x)-|E(D\setminus y)|+1+(t-1)(w''+1)+w_z]\\
&\geq & (\sum\limits_{x\in V(D)}w(x)-\sum\limits_{x\in V(D\setminus y)}w(x)-w_z)-(|E(D)|-|E(D\setminus y)|)\\
&\geq &(\ell-2+w_y)-\ell \geq 0,\hspace{8.0cm}\  \ (5)
\end{eqnarray*}
where the first inequality holds because of $w''\leq w$  and the second inequality holds because of $|N_{D}^{+}(y)\setminus \{z\}|=\ell-2$.
By formulas (3), (4) and  (5), we get
\[
\mbox{reg}\,((I(D)^t,z^{w_z}))\geq \mbox{max}\{\mbox{reg}\,((I(D)^{t}:yz^{w_z})(-w_z-1)),\mbox{reg}\,(((I(D)^{t}:z^{w_z}),y)(-w_z))\}.
\]
Using Lemma \ref{lem6} (2), (4) on
the short exact sequences (1), (2) and the equality (3),  we obtain
\[
\mbox{reg}\,(I(D)^{t})=
\mbox{reg}\,((I(D)^t,z^{w_z}))=\!\!\!\sum\limits_{x\in V(D)}\!\!\!w(x)-|E(D)|+1+(t-1)(w+1).
\]
The proof is completed.
\end{proof}

\medskip
 As a consequence of Theorem \ref{thm6}, we have
\begin{Corollary}\label{cor4}
Let $D=(V(D),E(D),w)$ be a vertex-weighted oriented unicyclic graph as  Theorem
\ref{thm6}. Then
$$\mbox{reg}\,(I(D)^{t})=\mbox{reg}\,(I(D))+(t-1)(w+1)\hspace{1cm}  \mbox{for any}\ \ t\geq 1,$$
where $w=\mbox{max}\,\{w(x)\mid x \in V(D)\}$.
\end{Corollary}

\medskip
The following  example shows  the assumption in Theorem \ref{thm6}
that $D$ is a vertex-weighted oriented unicyclic graph such that $w(x)\geq 2$ for any  $d(x)\neq 1$  cannot be dropped.
\begin{Example}  \label{example5}
Let $I(D)=(x_1x_2^2,x_2x_3^2,x_3x_4^2,x_4x_1^2,x_4x_5,x_5x_6,x_6x_7^2)$ be the edge ideal of an  oriented unicyclic graph, its  weight function is $w_1=w_2=w_3=w_4=w_7=2$ and $w_5=w_6=1$. Thus $w=2$. By using CoCoA, we obtain $\mbox{reg}\,(I(D)^2)=10$. But we have $\mbox{reg}\,(I(D)^2)=(\sum\limits_{i=1}^{7}w_i-|E(D)|+1)+(w+1)=9$  by Theorem \ref{thm6}.
\end{Example}

 The following example shows the regularity of  powers of
 edge ideals of  vertex-weighted  oriented  unicyclic graphs is related to
direction selection in Theorem \ref{thm6}.
\begin{Example}  \label{example6}
Let $I(D)=(x_1x_2^2,x_2x_3^2,x_3x_4^2,x_4x_1^2,x_4x_5^2,x_6x_5^2,x_6x_7^2)$ be the edge ideal of an oriented unicyclic graph, its  weight function is  $w_1=w_2=w_3=w_4=w_5=w_7=2$ and  $w_6=1$. Thus $w=2$. By using CoCoA, we obtain $\mbox{reg}\,(I(D)^2)=11$. But we have $\mbox{reg}\,(I(D)^2)=(\sum\limits_{i=1}^{7}w_i-|E(D)|+1)+(w+1)=10$  by Theorem \ref{thm6}.
\end{Example}

\medskip

\hspace{-6mm} {\bf Acknowledgments}

 \vspace{3mm}
\hspace{-6mm}  This research is supported by the National Natural Science Foundation of China (No.11271275) and  by foundation of the Priority Academic Program Development of Jiangsu Higher Education Institutions.

\end{document}